\documentclass[a4paper,pdftex,reqno,10pt]{amsart}

\allowdisplaybreaks[1]

\usepackage[german,english]{babel}
\usepackage{enumerate}
\usepackage{amsmath,amssymb}
\usepackage{geometry}

\newcommand{\subspaces}[2]{\genfrac{[}{]}{0pt}{}{#1}{#2}}
\newcommand{\vek}[1]{\mathbf{#1}} 
\newcommand{\abs}[1]{\lvert#1\rvert}
\newcommand{\mat}[1]{\mathbf{#1}}
\newcommand{\gauss}[3]{\genfrac{[}{]}{0pt}{}{#1}{#2}_{#3}}
\newcommand{\points}{\mathcal{P}} 
\DeclareMathOperator{\PG}{PG}
\DeclareMathOperator{\Aut}{Aut} 
\DeclareMathOperator{\GL}{GL}
\DeclareMathOperator{\GaL}{\Gamma L} 
\DeclareMathOperator{\PGL}{PGL}
\DeclareMathOperator{\PGaL}{P\Gamma L} 
\newcommand{\trace}{\mathrm{Tr}} 
\newcommand{\denum}{\delta}

\newcommand{\sdist}{\mathrm{d}_{\mathrm{S}}}
\newcommand{\idist}{\mathrm{d}_{\mathrm{I}}}
\newcommand{\smax}{\mathrm{A}}
\newcommand{\tp}{\mathsf{T}}
\newcommand{\prob}{\mathrm{p}}

\newcommand{\F}{\mathbb{F}} 
\newcommand{\R}{\mathbb{R}} 
\newcommand{\Z}{\mathbb{Z}}

\newcommand{\graph}{\Gamma}
\newcommand{\gabidulin}{\mathcal{G}} 

\usepackage[colorlinks=true]{hyperref} \hypersetup{urlcolor=blue,
  citecolor=red}

\newtheorem{theorem}{Theorem}[section]
\newtheorem{lemma}[theorem]{Lemma}
\newtheorem{question}{Question} 
\newtheorem*{problem}{Problem}

\theoremstyle{definition} 
\newtheorem{definition}[theorem]{Definition}
\newtheorem{remark}{Remark}

\title[Mixed-Dimension Subspace Codes]{Constructions and Bounds for
  Mixed-Dimension Subspace Codes}

\author[Thomas Honold, Michael Kiermaier and Sascha Kurz]{}

\subjclass{Primary 94B05, 05B25, 51E20; Secondary 51E14, 51E22, 51E23}
\keywords{Galois geometry, network coding, subspace code, partial spread}
\dedicatory{In memoriam Axel Kohnert (1962--2013)}
\email{honold@zju.edu.cn} 
\email{michael.kiermaier@uni-bayreuth.de}
\email{sascha.kurz@uni-bayreuth.de}

\thanks{Th.~Honold gratefully acknowledges financial
  support for a short visit to the University of Bayreuth in March
  ~2015, where part of this research was done. His work was also
  supported by the National Natural Science Foundation of China under
  Grant 61571006. M.~Kiermaier and S.~Kurz were supported by EU
  COST Action IC1104. In addition, S.~Kurz's work was supported
  through the DFG project ``Ganz\-zahlige Optimierungs\-modelle f\"ur
  Subspace Codes und endliche Geometrie'' (DFG grant KU 2430/3-1).}

\begin{document}
\maketitle


\centerline{\scshape Thomas Honold} \medskip {\footnotesize
  \centerline{Department of Information and Electronic Engineering}
  \centerline{Zhejiang University, 38 Zheda Road, 310027 Hangzhou,
    China} } 

\medskip

\centerline{\scshape Michael Kiermaier} \medskip {\footnotesize
  \centerline{Mathematisches Institut, Universit\"at Bayreuth, D-95440
    Bayreuth, Germany} }

\medskip

\centerline{\scshape Sascha Kurz} \medskip {\footnotesize
  \centerline{Mathematisches Institut, Universit\"at Bayreuth, D-95440
    Bayreuth, Germany} }

\bigskip


\begin{abstract}
  Codes in finite projective spaces equipped with the subspace
  distance have been proposed for error control in random
  linear network coding. The resulting so-called \emph{Main Problem
      of Subspace Coding} is to determine the maximum size
    $\smax_q(v,d)$ of a code in $\PG(v-1,\F_q)$ with minimum subspace
    distance $d$. Here we completely resolve this problem for
    $d\ge v-1$.  For $d=v-2$ we present some improved bounds and
    determine $\smax_q(5,3)=2q^3+2$ (all $q$), $\smax_2(7,5)=34$. We
    also provide an exposition of the known determination of
    $\smax_q(v,2)$, and a table with exact results and
    bounds for the numbers $\smax_2(v,d)$, $v\leq 7$.
\end{abstract}


\section{Introduction}\label{sec:intro}

For a prime power $q>1$ let $\mathbb{F}_q$ be the finite field with
$q$ elements and $\mathbb{F}_q^v$ the standard vector
space of dimension $v\geq 0$ over $\mathbb{F}_q$. The set of all subspaces
of $\mathbb{F}_q^v$, ordered by the incidence relation $\subseteq$, is
called \emph{($v-1$)-dimensional (coordinate) projective geometry over
  $\F_q$} and denoted by $\PG(v-1,\F_q)$. It forms a finite modular geometric
lattice with meet $X\wedge Y=X\cap Y$ and join $X\vee Y=X+Y$.

The study of geometric and combinatorial properties of $\PG(v-1,\F_q)$
and related structures forms the subject of \emph{Galois Geometry}---a
mathematical discipline with a long and renowned history of its own
but also with links to several other areas of discrete mathematics and
important applications in contemporary industry, such as cryptography
and error-correcting codes. For a comprehensive introduction to the
core subjects of Galois Geometry readers may consult the three-volume
treatise \cite{hirschfeld85,hirschfeld98,hirschfeld-thas91}. More
recent developments are surveyed in \cite{nova2011}.

It has long been recognized that classical error-correcting codes,
which were designed for point-to-point communication over a period of
now more than 60 years, can be studied in the Galois Geometry
framework. Recently, through the seminal work of Koetter, Kschischang
and Silva
\cite{koetter-kschischang08,silva-kschischang-koetter08,silva-kschischang-koetter10},
it was discovered that essentially the same is true for the
network-error-correcting codes developed by Cai, Yeung, Zhang and
others\ \cite{yeung-cai06a,yeung-cai06b,guang-zhang14}. Information in
packet networks with underlying packet space $\F_q^v$ can be
transmitted using subspaces of $\PG(v-1,\F_q)$ as codewords and
secured against errors (both random and adversarial errors) by
selecting the codewords subject to a lower bound on their mutual
distance in a suitable metric on $\PG(v-1,\F_q)$, resembling the
classical block code selection process based on the Hamming distance
properties.


Accordingly, we call any set $\mathcal{C}$ of subspaces of $\F_q^v$ a
\emph{$q$-ary subspace code of packet length $v$}.  Two widely used
distance measures for subspace codes (motivated by an
information-theoretic analysis of the Koetter-Kschischang-Silva model)
are the so-called \textit{subspace distance}
\begin{equation}
  \begin{aligned}
    \sdist(X,Y)&=\dim(X+Y)-\dim(X\cap Y)\\
    &=\dim(X)+\dim(Y)-2\cdot\dim(X\cap Y)\\
    &=2\cdot\dim(X+Y)-\dim(X)-\dim(Y)
  \end{aligned}  
\end{equation} 
and \textit{injection distance}
\begin{equation}
  \idist(X,Y)= \max\left\{\dim(X),\dim(Y)\right\}-\dim(X\cap Y).
\end{equation} 
With this the \textit{minimum distance} in the subspace metric of a
subspace code $\mathcal{C}$ containing at least two codewords is defined as
\begin{equation}
  \sdist(\mathcal{C}):=\min\left\{\sdist(X,Y);X,Y\in\mathcal{C},
    X\neq Y\right\},
\end{equation}
and that in the injection metric as
\begin{equation}
  \idist(\mathcal{C}):=\min\left\{\idist(X,Y);X,Y\in\mathcal{C},
    X\neq Y\right\}.\footnotemark 
\end{equation}
\footnotetext{Sometimes it will be convenient to allow
  $\#\mathcal{C}\leq 1$, in which case we formally set
  $\sdist(\mathcal{C})=\idist(\mathcal{C})=\infty$.}%
A subspace code $\mathcal{C}$ is said to be a
\textit{constant-dimension code} (or \textit{Grassmannian code}) if all
codewords in $\mathcal{C}$ have the same dimension over $\F_q$. Since
$\sdist(X,Y)=2\cdot\idist(X,Y)$ whenever $X$ and $Y$ are of the same
dimension, we need not care about the specific metric ($\sdist$ or
$\idist$) used when dealing with constant-dimension codes. Moreover,
$\sdist(\mathcal{C})=2\cdot\idist(\mathcal{C})$ in this case and hence
the minimum subspace distance of a constant-dimension code
is always an even integer.  

As in the classical case of block codes, the transmission rate of a
network communication system employing a subspace code $\mathcal{C}$
is proportional to $\log(\#\mathcal{C})$. Hence, given a lower bound
on the minimum distance $\sdist(\mathcal{C})$ or
$\idist(\mathcal{C})$ (providing, together with other parameters such
as the physical characteristics of the network and the decoding
algorithm used, a specified data integrity level\footnote{This
  integrity level is usually specified by an upper bound on the
  probability of transmission error allowed.}), we want the code size
$M=\#\mathcal{C}$ to be as large as possible. It is clear that
constant-dimension codes usually are not maximal in this respect and,
as a consequence, we need to look at general mixed-dimension subspace
codes for a rigorous solution of this optimization problem.

In the remaining part of this article we will
restrict ourselves to the subspace distance $\sdist$.
From a mathematical point of
view, any $v$-dimensional vector space $V$ over $\F_q$ is just as good
as the standard space $\F_q^v$ (since $V\cong\F_q^v$), and it will sometimes be
convenient to work with non-standard spaces (for example with the
extension field $\F_{q^v}/\F_q$, in order to exploit additional structure).
Hence we fix the following terminology:

\begin{definition}
  A \emph{$q$-ary $(v,M,d)$ subspace code}, also referred to as a
  \emph{subspace code with parameters $(v,M,d)_q$}, is a set
  $\mathcal{C}$ of subspaces of $V\cong\F_q^v$ with $M=\#\mathcal{C}$ and
  $\sdist(\mathcal{C})=d$.
  The space $V$ is called the \emph{ambient space} of
  $\mathcal{C}$.\footnote{Strictly speaking, the ambient space is part
    of the definition of a subspace code and we should write
    $(V,\mathcal{C})$ in place of $\mathcal{C}$. Since the ambient
    space is usually clear from the context, we have adopted the
    more convenient shorthand ``$\mathcal{C}$''.} The \emph{dimension
    distribution} of $\mathcal{C}$ is the sequence
  $\denum(\mathcal{C})=\left(\delta_0,\delta_1,\dots,\delta_v\right)$
  defined by $\delta_k=\#\{X\in\mathcal{C};\dim(X)=k\}$. Two subspace
  codes $\mathcal{C}_1,\mathcal{C}_2$ are said to be \emph{isomorphic}
  if there exists an isometry (with respect to the subspace metric)
  $\phi\colon V_1\to V_2$ between their ambient spaces satisfying
  $\phi(\mathcal{C}_1)=\mathcal{C}_2$.
\end{definition}
It is easily seen that isomorphic subspace codes
$\mathcal{C}_1,\mathcal{C}_2$ must have the same
alphabet size $q$ and the same ambient space dimension
$v=\dim(V_1)=\dim(V_2)$. The dimension distribution of a subspace 
code may be seen as a $q$-analogue of the Hamming weight distribution
of an ordinary block code. As in the block code case, the quantities
$\delta_k=\delta_k(\mathcal{C})$ are non-negative integers satisfying
$\sum_{k=0}^v \delta_k=M=\#\mathcal{C}$.
\begin{problem}[Main Problem of Subspace Coding\footnote{The name
    has been chosen to emphasize the analogy 
    with the ``main problem of (classical) coding theory'', which
    commonly refers to the problem of optimizing the parameters
    $(n,M,d)$ of ordinary block codes \cite{macwilliams-sloane77},
    and should not be taken literally as meaning ``the most
    important problem in this area''.}]
  For a given prime power $q\geq 2$, packet length $v\geq 1$ and
  minimum distance $d\in\{1,\dots,v\}$ determine the maximum size
  $\smax_q(v,d)=M$ of a $q$-ary $(v,M,d)$ subspace code 
  and---as a refinement---classify the
  corresponding optimal codes up to subspace code isomorphism.
\end{problem}


Although our ultimate focus will be on the main problem for general
mixed-dimension subspace codes as indicated, we will often build upon
known results for the same problem restricted to constant-dimension
codes, or mixed-dimension codes with only a small number of nonzero
dimension frequencies $\delta_k$. For this it will be convenient to
denote, for subsets $T\subseteq\{0,1,\dots,v\}$, the maximum size
of a $(v,M,d')_q$ subspace code $\mathcal{C}$ with $d'\geq d$ and
$\delta_k(\mathcal{C})=0$ for all $k\in\{0,1,\dots,v\}\setminus T$
by $\smax_q(v,d;T)$, and refer to subspace codes subject to this
dimension restriction accordingly as $(v,M,d;T)_q$
codes.
In other words, the set $T$ specifies the dimensions of the subspaces which can
be chosen as a codeword of a $(v,M,d;T)_q$ code, and determining the numbers 
$\smax_q(v,d;T)$ amounts to extending the main problem to $(v,M,d;T)_q$
codes.\footnote{In order to make $\smax_q(v,d;T)$ well-defined for all
  $d\in\{1,\dots,v\}$ and ensure the usual monotonicity property
  $\smax_q(v,d;T)\geq\smax_q(v,d';T)$ for $d\leq d'$,
  it is necessary to take the maximum in the definition of
$\smax_q(v,d;T)$ over all codes with $d'\geq d$ (and not only over
codes with exact minimum distance $d$). The definition of
$\smax_q(v,d)$ didn't require such extra care, since in the
unrestricted case we can alter dimensions of codewords freely and
hence transform any $(v,M,d')_q$ code with $d'\geq d$ into a
$(v,M,d)_q$ code.}


While a lot of research has been done on the
determination of the numbers $\smax_q(v,d;k)=\smax_q(v,d;\{k\})$, the
constant-dimension case, only very few results are known for
$\#T>1$.\footnote{Strictly speaking, this remark is true only in the
  binary case. For $q>2$ even in the
  constant-dimension case very few results are known.}

The purpose of this paper is to advance the knowledge in the
mixed-dimension case and determine the numbers
$\smax_q(v,d)$ for further parameters $q$, $v$, $d$. In this
regard we build upon previous work of several authors, as is nicely
surveyed by Etzion in
\cite[Sect.~4]{etzion2013problems}. In the parlance of
Etzion's survey, our contribution partially solves Research
Problem~20.\footnote{For those already familiar
with \cite{etzion2013problems} we remark that our numbers
$\smax_q(v,d)$ translate into Etzion's $\mathcal{A}_q^S(v,d)$.}


The Gaussian binomial coefficients $\gauss{v}{k}{q}$
give the number of $k$-dimensional subspaces of $\F_q^v$ (and of any
ambient space $V\cong\F_q^v$)
and satisfy
\begin{equation}
  \label{eq:gauss}
  \gauss{v}{k}{q}=\prod_{i=0}^{k-1}\frac{q^{v-i}-1}{q^{k-i}-1}
  =q^{k(v-k)}\cdot\bigl(1+\mathrm{o}(1)\bigr)\quad\text{for $q\to\infty$}.
\end{equation}
Since these numbers grow very quickly, especially for $k\approx v/2$,
the exact determination of $\smax_q(v,d)$ appears to be an intricate
task---except for some special cases. Even more challenging is the
refined problem of enumerating the isomorphism types of the
corresponding optimal subspace codes (i.e.\ those of size
$\smax_q(v,d)$). 
In some cases such an exhaustive enumeration is currently infeasible
due to the large number of isomorphism types or due to
computational limitations. In this context we regard the determination
of certain structural restrictions as a precursor to an exhaustive
classification.



The remaining part of this paper is structured as follows.  In
Section~\ref{sec:prelim} we provide further terminology and a few
auxiliary concepts and results, which have proved useful for the
subsequent subspace code optimization/classification. In
Section~\ref{sec:general} we determine the numbers $\smax_q(v,d)$ for
general $q$, $v$ and some special values of $d$. Finally, in
Section~\ref{sec:special} we further discuss the binary case $q=2$
and determine the numbers $\smax_2(v,d)$, as well as the corresponding
number of isomorphism classes of optimal codes, for $v\leq 7$ and all but a
few hard-to-resolve cases. For the remaining cases we provide improved
bounds for $\smax_2(v,d)$ using a variety of methods.
On the reader's side we will assume at least some rudimentary
knowledge of subspace coding, which e.g.\ can be acquired by reading
the survey \cite{etzion2015galois} or its predecessor
\cite{etzion2013problems}.

\section{Preliminaries}
\label{sec:prelim}

\subsection{The Automorphism Group of $\bigl(\PG(V),\sdist\bigr)$}
\label{ssec:auto}

Let us start with a description of the automorphism group 
of the metric space $\PG(v-1,\F_q)$ relative to the subspace
distance. Since a general $v$-dimensional ambient space $V$ is
isomorphic to $\F_q^v$ as a vector space over $\F_q$ (and hence
isometric to $(\F_q^v,\sdist)$), this yields a description of all automorphism
groups $\Aut(V,\sdist)$ and also of all isometries between different 
ambient spaces $(V_1,\sdist)$ and $(V_2,\sdist)$.

It is clear that the linear group $\GL(v,\F_q)$ acts on
$\PG(v-1,\F_q)$ as a group of $\F_q$-linear isometries. If $q$ is not
prime then there are additional semilinear isometries arising from the
Galois group $\Aut(\F_q)=\Aut(\F_q/\F_p)$ in the obvious way
(component-wise action on $\F_q^v$). Moreover, mapping a subspace
$X\subseteq\F_q^v$ (``linear code of length $v$ over $\F_q$'') to its
dual code $X^\perp$ (with respect to the standard inner product)
respects the subspace distance and hence yields a further automorphism
$\pi$ of the metric space $\PG(v-1,\F_q)$. The map $\pi$ also
represents a polarity (correlation of order $2$) of the geometry
$\PG(v-1,\F_q)$.

\begin{theorem}
  \label{thm:autos}
  Suppose that $v\geq 3$.\footnote{The case $v\leq 2$ is completely
    trivial---as far as subspace codes are concerned---and can be
    safely excluded.} The automorphism group $G$ of
  $\PG(v-1,\F_q)$, viewed as a metric space with respect to the
  subspace distance, is generated by $\GL(v,\F_q)$,
  $\Aut(\F_q)$ and $\pi$. More precisely, $G$ is the
  semidirect product of the projective general semilinear group
  $\PGaL(v,\F_q)$ with a group of order $2$ acting by matrix
  transposition on $\PGL(v,\F_q)$ and trivially on $\Aut(\F_q)$.  
\end{theorem}
Most of this theorem is already contained in \cite{trautmann13}, but
we include a complete proof for convenience.
\begin{proof}
  Let $f$ be an automorphism of $\PG(v-1,\F_q)$. Then either $f$
  interchanges $\{\vek{0}\}$ and $\F_q^v$ or leaves both subspaces
  invariant (using the fact that $\{\vek{0}\}$, $\F_q^v$ are the only
  subspaces with a unique complementary subspace). Moreover, if $f$
  fixes $\{\vek{0}\}$, $\F_q^v$ then it preserves the dimension of
  subspaces and hence represents a collineation of the geometry
  $\PG(v-1,\F_q)$. By the Fundamental Theorem of Projective Geometry
  (here we use the assumption $v\geq 3$), a collineation is
  represented by an element of $\PGaL(v,\F_q)$. Since $\pi$
  interchanges $\{\vek{0}\}$ and $\F_q^v$, either $f$ or $f\circ\pi$
  stabilizes $\{\vek{0}\}$, $\F_q^v$ and belongs to
  $\PGaL(v,\F_q)$. This proves the first assertion and shows that
  $\PGaL(v,\F_q)$ has index $2$ in $G$.\footnote{The nontrivial coset
    $\{\pi\circ g;g\in\PGaL(v,\F_q)\}$ consists precisely of all
    correlations of $\PG(v-1,\F_q)$.}  Finally, denoting by $\phi$ the
  Frobenius automorphism of $\F_q$ (over its prime field $\F_p$), we
  have
  $\phi(x_1y_1+\dots+x_vy_v)=\phi(x_1)\phi(y_1)+\dots+\phi(x_v)\phi(y_v)$
  and hence, using a dimension argument,
  $\phi(X^\perp)=\phi(X)^\perp$, i.e.\
  $\phi\circ\pi=\pi\circ\phi$. Since the adjoint map (with respect to
  the standard inner product on $\F_q^v$) of
  $\vek{x}\to\mat{A}\vek{x}$ is $\vek{y}\to\mat{A}^\tp\vek{y}$, the
  second assertion follows and the proof is complete.
\end{proof}
In effect, Theorem~\ref{thm:autos} reduces the isomorphism problem for
subspace codes to the determination of the orbits of $\GL(v,\F_q)$,
respectively $\GaL(v,\F_q)$, on subsets of $\PG(v-1,\F_q)$. In the
most important case $q=2$ the semilinear part is void, which further
simplifies the problem. As a word of caution we remark that, in view
of the presence of the polarity $\pi$, the dimension distribution of
subspace codes is not an isomorphism invariant. Rather we have that
$\delta(\mathcal{C})=(\delta_0,\delta_1,\dots,\delta_v)$ leaves for a
code $\mathcal{C}'\cong\mathcal{C}$ the reverse distribution
$\delta(\mathcal{C}')=(\delta_v,\delta_{v-1},\dots,\delta_0)$ as a
second possibility.

The formula in \eqref{eq:gauss} for the Gaussian binomial
coefficients, which can be put into the form
\begin{align*}
  \gauss{v}{k}{q}&=\frac{\prod_{i=0}^{v-1}(q^v-q^i)}
  {q^{k(v-k)}\cdot\prod_{i=0}^{k-1}(q^k-q^i)
  \cdot\prod_{i=0}^{v-k-1}(q^{v-k}-q^i)}\\
  &=\frac{\#\GL(v,\F_q)}{q^{k(v-k)}\cdot\#\GL(k,\F_q)\cdot\#\GL(v-k,\F_q)},
\end{align*}
reflects the group-theoretical fact that $\GL(v,\F_q)$ acts on the set
of $k$-dimensional subspaces of $\F_q^v$ transitively and
with a stabilizer isomorphic to
\begin{equation*}
  \begin{pmatrix}
    \GL(k,\F_q)&*\\
    \mat{0}&\GL(v-k,\F_q)
  \end{pmatrix}.
\end{equation*}
We can go further and ask for a description of the orbits of
$\GL(v,\F_q)$ in its induced action on ordered pairs $(X,Y)$ of
subspaces. Such a description has significance for modeling the
transmission of subspaces in the Koetter-Kschischang-Silva model by a
discrete memoryless (stationary) channel, which in essence amounts to
specifying time-independent transition probabilities $\prob(Y|X)$.


\begin{lemma}
  \label{lemma_iso_pairs}
  For any integer triple $a,b,c$ satisfying $0\leq a,b\leq v$ and
  $\max\{0,a+b-v\}\leq c\leq\min\{a,b\}$ the group $\GL(v,q)$ acts
  transitively on ordered pairs of subspaces $(X,Y)$ of $\F_q^v$ with
  $\dim(X)=a$, $\dim(Y)=b$, and $\dim(X\cap
  Y)=c$.\footnote{Alternatively, we could prescribe the dimension of
    the join $X+Y$ in place of $X\cap Y$.} Moreover, each such
  integer triple gives rise to an orbit of $\GL(v,\F_q)$ on ordered
  pairs of subspaces of $\F_q^v$ with length
  \begin{equation*}
    q^{(a-c)(b-c)}\gauss{v}{c}{q}\gauss{v-c}{a-c}{q}\gauss{v-a}{b-c}{q}>0.
  \end{equation*}

\end{lemma}
\begin{proof}
  The restrictions on $a,b,c$ are obviously necessary, since $\dim(X\cap
  Y)\leq\min\{\dim(X),\dim(Y)\}$ and $\dim(X\cap
  Y)=\dim(X)+\dim(Y)-\dim(X+Y)\geq\dim(X)+\dim(Y)-v$. Conversely, if
  $a,b,c$ satisfy the restrictions then $c$, $a-c$, $b-c$ are
  non-negative with sum $c+(a-c)+(b-c)=a+b-c\leq v$. Hence we can choose
  $a+b-c$ linearly independent vectors
  $\vek{b}_1,\dots,\vek{b}_{a+b-c}$ in $\F_q^v$ and set
  $X=\langle\vek{b}_1,\dots,\vek{b}_a\rangle$,
  $Y=\langle\vek{b}_1,\dots,\vek{b}_b\rangle$, and consequently $X\cap
  Y=\langle\vek{b}_1,\dots,\vek{b}_c\rangle$.

  It remains to show that $\GL(v,q)$ acts transitively on those pairs
  of subspaces and compute the orbit lengths. Transitivity is an
  immediate consequence of the fact that the corresponding sequences
  of $a+b-c$ linearly independent vectors, defined as above, can be
  isomorphically mapped onto each other. The stabilizer of $(X,Y)$ in
  $\GL(v,\F_q)$ has the form
  \begin{equation*}
    \begin{pmatrix}
      \GL(c,\F_q)&*&*&*\\
      \mat{0}&\GL(a-c,\F_q)&\mat{0}&*\\
      \mat{0}&\mat{0}&\GL(b-c,\F_q)&*\\
      \mat{0}&\mat{0}&\mat{0}&\GL(v-a-b+c,\F_q)
    \end{pmatrix},
  \end{equation*}
  which leads to the stated formula for the orbit length
  after a short computation.\footnote{Alternatively, count quadruples
    $(X,Y,Z,W)$ of subspaces of $\F_q^v$ satisfying $X\cap Y=Z$,
    $X+Y=W$ and $X/Z$ is complementary to $Y/Z$ in $W/Z$.}
\end{proof}


\subsection{Basic Properties of the Numbers $\smax_q(v,d;T)$}
\label{ssec:basic}

In this subsection we collect some elementary but useful properties of
the numbers $\smax_q(v,d;T)$ and consider briefly the growth of
$k\mapsto\smax_q(v,d;k)$ (the constant-dimension case). Henceforth
$V$ will denote a $v$-dimensional vector space over $\F_q$, if not
explicitly stated otherwise, and we will use the abbreviations
$\subspaces{V}{T}$ for the set of all subspaces $X\subseteq V$ with
$\dim(X)\in T$ ($\subspaces{V}{k}$ in the constant-dimension case
$T=\{k\}$) and $\mathcal{C}_T=\mathcal{C}\cap\subspaces{V}{T}$ for
subspace codes $\mathcal{C}$ with ambient space $V$ (with the usual
convention $\mathcal{C}_k=\mathcal{C}_{\{k\}}$). Further we set
$[a,b]=\{a,a+1,\dots,b-1,b\}$ for $0\leq a\leq b\leq v$.

Note that the following
properties apply in particular to $\smax_q(v,d)=\smax_q(v,d;[0,v])$.

\begin{lemma}\label{lma:smax}\label{obs_mixed_bounds}
  \begin{enumerate}[(i)]
  \item\label{smax:d=1}
    $\smax_q(v,1;T)=\sum_{t\in T} \gauss{v}{t}{q}$, and the unique
    optimal code in this case is
    $\subspaces{V}{T} =\biguplus_{t\in T}\subspaces{V}{t}$;
  \item\label{smax:d<d'}
    $\smax_q(v,d;T)\geq\smax_q(v,d';T)$ for all $1\leq d\leq d'\leq v$;
  \item\label{smax:subset} 
    $\smax_q(v,d;T)\leq\smax_q(v,d;T')$ for all
    $T\subseteq T'\subseteq[0,v]$;
  \item\label{smax:subadd}
    $\smax_q(v,d;T\cup T')\le \smax_q(v,d;T)+\smax_q(v,d;T')$ for all
    $T, T'\subseteq[0,v]$; equality holds if
    $\min\bigl\{\abs{t-t'};t\in T,t'\in T'\bigr\}$ (i.e., the distance
    between $T$ and $T'$ in the Euclidean metric) is at least $d$;
  \item\label{smax:symm}
    $\smax_q(v,d;T)=\smax_q(v,d;v-T)$, where $v-T=\{v-t;t\in T\}$.
  \item\label{smax:diam} 
    The metric space $\subspaces{V}{T}$, $T\neq\emptyset$, has
    diameter $v-d$, where $d=\min\bigl\{\abs{s+t-v};s,t\in T\bigr\}$
    (the distance between $v$ and $T+T\subset\R$ in
    the Euclidean metric).\footnote{In particular, $\subspaces{V}{T}$
      has diameter $v$ if there exist $s,t\in T$ with $s+t=v$ and
      diameter $<v$ otherwise.}
  \end{enumerate}
\end{lemma}
\begin{proof}
  Only \eqref{smax:subadd}, \eqref{smax:symm} and \eqref{smax:diam}
  require a proof. 

  In \eqref{smax:subadd} we may assume $T\cap T'=\emptyset$.
  (Otherwise write $T\cup T'=T\uplus(T'\setminus T)$ and use
  \eqref{smax:subset}.) If $\mathcal{C}$ is any $(v,M,d;T\cup T')_q$
  code then $\#\mathcal{C}_T\leq\smax_q(v,d;T)$,
  $\#\mathcal{C}_{T'}\leq\smax_q(v,d;T')$ and
  $M=\#\mathcal{C}_T+\#\mathcal{C}_{T'}\leq\smax_q(v,d;T)
  +\smax_q(v,d;T')$, as asserted.

  For the proof of \eqref{smax:symm} assume $V=\F_q^v$ and use the
  fact that the map $\pi\colon X\to X^\perp$ represents an
  automorphism of the metric space $\PG(v-1,\F_q)$ and maps
  $\subspaces{V}{T}$ onto $\subspaces{V}{v-T}$.

  Finally, \eqref{smax:symm} implies that the largest possible
  distance between $X\in\subspaces{V}{s}$, $Y\in\subspaces{V}{t}$ is
  $\min\{s+t,2v-s-t\}$. (This is clearly true if $s+t\leq v$, and the
  case $s+t>v$ can be reduced to the former by setting $s'=v-s$,
  $t'=v-t$ and using \eqref{smax:symm}.) In particular, the diameter
  of $\subspaces{V}{s}$ is $2\min\{s,v-s\}$. Assertion
  \eqref{smax:diam} now follows from the observation that
  $\min\{s+t,2v-s-t\}=v-\abs{s+t-v}$.
\end{proof}

Next we discuss the growth of the numbers
$\smax_q(v,d;k)$ as a function of $k\in\{0,1,\dots,\lfloor
v/2\rfloor\}$.\footnote{By symmetry, the range
  $k\in\{\lfloor v/2\rfloor+1,\dots,v\}$ need not be considered; cf.\
  Lemma~\ref{lma:smax}\eqref{smax:symm}.} While not directly applicable to the
mixed-dimension case, this analysis provides some useful information also
for this case, since mixed-dimension codes are composed
of constant-dimension ``layers''.

Since the minimum distance of a
a constant-dimension code is an even integer, we need only consider
the case $d=2\delta\in2\Z$.
\begin{lemma}
  \label{lma:unimodal}
  For $1\leq\delta\leq k\leq\lfloor v/2\rfloor$ the inequality
  \begin{equation*}
  \frac{\smax_q(v,2\delta;k)}{\smax_q(v,2\delta;k-1)}
  >q^{v-2k+\delta}\cdot C(q,\delta)
\end{equation*}
holds with $C(q,1)=1$ and $C(q,\delta)=1-1/q$ for $\delta\geq 2$;
in particular, $\smax_q(v,2\delta;k)>q\cdot
\smax_q(v,2\delta;k-1)$. As a consequence, the numbers $\smax_q(v,2\delta;k)$,
$k\in[\delta,v-\delta]$, form a strictly unimodal sequence.
\end{lemma}
Note that $C(q,\delta)$ is independent of $v,k$ and satisfies
$\lim_{q\to\infty}C(q,\delta)=1$. In fact our proof of the lemma will show that
for $\delta\geq 2$ the number
$C(q,\delta)=1-q^{-1}$ may be replaced by the larger quantity 
$\prod_{i=\delta}^\infty(1-q^{-i})$, which is even closer to $1$.
\begin{proof}
  First we consider the case $\delta=1$, in which the numbers 
  $\smax_q(v,2\delta;k)=\smax_q(v,2;k)=\gauss{v}{k}{q}$ are already known. 
  Here the assertion follows from
  \begin{equation*}
    \frac{\gauss{v}{k}{q}}{\gauss{v}{k-1}{q}}
    =\frac{q^{v-k+1}-1}{q^k-1}
    =q^{v-2k+1}\cdot\frac{1-q^{-(v-k+1)}}{1-q^{-k}}
    >q^{v-2k+1},
  \end{equation*}
  using $v-k+1>k$ for the last inequality.

  Now assume $\delta\geq 2$.
  The lifting construction produces
  $(v,q^{(k-\delta+1)(v-k)},2\delta;k)_q$ constant-dimension codes
    (``lifted MRD codes'') and gives the bound
    $\smax_q(v,d;k)\geq q^{(k-\delta+1)(v-k)}$, which is enough for our
      present purpose. On the other hand, every $(v,M,2\delta;k)_q$ code
      satisfies $\sdist(X,X')=2k-2\dim(X\cap X')\geq 2\delta$ or,
      equivalently, $\dim(X\cap X')\leq k-\delta$ for any two distinct
      codewords $X,X'\in\mathcal{C}$. This says that
      ($k-\delta+1$)-dimensional subspaces of $V$ are contained in at
      most one codeword of $\mathcal{C}$ and gives by double-counting
      the upper bound
      \begin{equation}
        \label{eq:schonheim}
        M\leq\frac{\gauss{v}{k-\delta+1}{q}}{\gauss{k}{k-\delta+1}{q}}
        =\frac{(q^v-1)(q^{v-1}-1)\dotsm(q^{v-(k-\delta)}-1)}
        {(q^k-1)(q^{k-1}-1)\dotsm(q^{\delta}-1)}
      \end{equation}
      for such codes. 
      Simplifying we get 
      $M<q^{(k-\delta+1)(v-k)}/C(q,\delta,k)$ with
      $C(q,\delta,k)=\prod_{i=\delta}^k(1-q^{-i})$. Replacing $k$ by
      $k-1$ turns this into an upper bound for
      $\smax_q(v,2\delta,k-1)$ and, together with the previously
      derived lower bound for $\smax_q(v,d;k)$, gives the estimate
      \begin{equation}
        \label{eq:quotient}
      \begin{aligned}
        \frac{\smax_q(v,2\delta;k)}{\smax_q(v,2\delta;k-1)}
        &>\frac{q^{(k-\delta+1)(v-k)}
          \cdot C(q,\delta,k)}{q^{(k-\delta)(v-k+1)}}
        =q^{v-2k+\delta}\cdot C(q,\delta,k)\\
        &>q^{v-2k+\delta}\cdot\prod_{i=\delta}^\infty(1-q^{-i})
          =q^{v-2k+\delta}\cdot
          \frac{\prod_{i=1}^\infty(1-q^{-i})}{\prod_{i=1}^{\delta-1}(1-q^{-i})}.
      \end{aligned}
      \end{equation}
      From Euler's Pentagonal Number Theorem
      (see e.g.\ \cite[Th.~15.5]{lint-wilson92}) we have
      \begin{align*}
        \prod_{i=1}^\infty(1-q^{-i})&=1+\sum_{m=1}^\infty(-1)^m\left(q^{-(3m^2-m)/2}
        +q^{-(3m^2+m)/2}\right)\\
        &=1-q^{-1}-q^{-2}+q^{-5}+q^{-7}-q^{-12}-q^{-15}\pm\dotsb\\
        &>1-q^{-1}-q^{-2}.
      \end{align*}
      Hence (and using
      $\delta\geq 2$), the quotient in \eqref{eq:quotient} is
      $>\frac{1-q^{-1}-q^{-2}}{1-q^{-1}}=\frac{q^2-q-1}{q^2-q}=1-\frac{1}{q(q-1)}
      \geq1-\frac{1}{q}$, as claimed.
      The remaining assertions of the lemma are clear.
\end{proof}

From the lemma, the numbers $\smax_q(v,d;k)$
grow fast as a function of $k$ in the range $0\leq k\leq
v/2$.  This implies that the following simple estimates yield quite a
good approximation to $\smax_q(v,d)$.

\begin{theorem}
  \label{thm:bounds}
  Suppose that for some parameters $q,v,d$ 
  we already know all of the numbers $\smax_q(v,d;k)$, $0\leq k\leq v$.
  Then
  \begin{equation*}
    \sum_{\substack{k=0\\k\equiv\lfloor v/2\rfloor\bmod d}}^v
      \smax_q\bigl(v,2\lceil d/2\rceil;k\bigr)
    \leq\smax_q(v,d)\leq 2+\sum_{k=\lceil d/2\rceil}^{v-\lceil d/2\rceil}\smax_q
    \bigl(v,2\lceil d/2\rceil;k\bigr),
  \end{equation*}
and this constitutes
the best bound for $\smax_q(v,d)$ that does not depend on
information about the cross-distance distribution between different
layers $\subspaces{V}{k}$ and $\subspaces{V}{l}$.
\end{theorem}
\begin{proof}
  First note that $2\delta$, where
  $\delta=\lceil d/2\rceil$, is the smallest even integer $\geq d$
  and hence $\smax_q(v,d;k)=\smax_q(v,2\delta;k)$. Now the upper
  bound follows from the observation that the two
  (isomorphic) metric spaces consisting of
  all subspaces of $V$ of dimension $<\delta$
  (respectively, $>v-\delta$) have diameter $<d$ and thus contain
  at most one codeword of any $(v,M,d)_q$ code.

  The lower bound follows from the inequality
  $\sdist(X,Y)\geq\abs{\dim(X)-\dim(Y)}$ and remains valid if we replace
  $\lceil v/2\rceil$ by an arbitrary integer $r$. In order to show that
  the lower bound is maximized for $r=\lceil d/2\rceil$, let
  $\sigma_r$ denote the sum of all numbers
  $\smax_q(v,2\delta;k)$ with $k\in[0,v]$ and
  $k\equiv r\bmod d$. Since $\sigma_r$ is $d$-periodic and satisfies
  $\sigma_r=\sigma_{v-r}$ for $0\leq r\leq v$, it suffices to show 
  $\sigma_r>\sigma_{r-1}$ for $\left\lceil(v-d)/2\right\rceil+1\leq r
  \leq\left\lfloor v/2\right\rfloor$. 

  For $r$ in the indicated range, $\sigma_r-\sigma_{r-1}$ is a sum of terms
  of the form
  \begin{multline*}
    \smax_q(v,2\delta;r-td)-\smax_q(v,2\delta;r-1-td)
    +\smax_q\bigl(v,2\delta;r+(t+1)d\bigr)-\smax_q\bigl(v,2\delta;r-1
    +(t+1)d\bigr)=\\
    =\smax_q(v,2\delta;r-td)-\smax_q(v,2\delta;r-1-td)
    -\smax_q\bigl(v,2\delta;v-r+1-(t+1)d\bigr)
    +\smax_q\bigl(v,2\delta;r+(t+1)d\bigr),
  \end{multline*}
  where $0\leq t\leq\lfloor r/d\rfloor$ and the convention
  $\smax_q(v,2\delta;k)=0$ for $k\notin[0,v]$ has been used. From
  Lemma~\ref{lma:unimodal} we have
  \begin{equation*}
    \smax_q(v,2\delta;r-td)
    \begin{cases}
      >q\cdot\smax_q(v,2\delta;r-1-td),\\
      \geq q^{2r-v-1+d}\cdot
      \smax_q\bigl(v,2\delta;v-r+1-(t+1)d\bigr),
    \end{cases}
  \end{equation*}
  and $2r-v-1+d\geq d-1\geq 1$.\footnote{Of course this implies that
    the second inequality above is also strict. The trivial case $d=1$ has
    been tacitly excluded.} From this (and $q\geq 2$) we can certainly
  conclude that
  $\smax_q(v,2\delta;r-td)
  >\smax_q(v,2\delta;r-1-td)+\smax_q\bigl(v,2\delta;v-r+1-(t+1)d\bigr)$,
  so that $\sigma_r-\sigma_{r-1}$ is positive, as claimed.
\end{proof}

\subsection{Shortening and Puncturing Subspace Codes}
\label{ssec:shortpunct}

In \cite{koetter-kschischang08,etzion2013problems} two different 
constructions of $(v-1,M',d')_q$ subspace
codes from $(v,M,d)_q$ subspace codes were defined and both referred to as
``puncturing subspace codes''. Whereas the construction in
\cite{koetter-kschischang08} usually has $M'=M$ (as is the case for
puncturing block codes), the construction in \cite{etzion2013problems}
satisfies $M'<M$ apart from trivial cases and behaves very much like
the shortening construction for block codes. For this reason, we
propose to change its name to ``shortening subspace codes''. We will now
give a simple, coordinate-free definition of the shortening
construction and generalize the puncturing construction
of \cite{koetter-kschischang08} to incorporate simultaneous
point-hyperplane puncturing.

\begin{definition}
  \label{def:shortening}
  Let $\mathcal{C}$ be a subspace code with ambient space $V$, $H$ a
  hyperplane and $P$ a point of $\PG(V)$. The \emph{shortened codes} of
  $\mathcal{C}$ in $H$, $P$ and the pair $P,H$ are defined as
  \begin{align*}
    \mathcal{C}|_{H}&=\{X\in\mathcal{C};X\subseteq H\},\\
    \mathcal{C}|^{P}&=\{X/P;X\in\mathcal{C},P\subseteq X\},\\
    \mathcal{C}|_H^P&=\{X\in\mathcal{C};X\subseteq H\}\cup\{Y\cap
    H;Y\in\mathcal{C},P\subseteq Y\}\\
                    &=\mathcal{C}|_{H}\cup\{Y\cap H;Y\in\mathcal{C}|_P\}
  \end{align*}
   with ambient spaces $H$, $V/P$ and $H$, respectively.
\end{definition}
Note that the operations $\mathcal{C}\mapsto\mathcal{C}|^P$ and
$\mathcal{C}\mapsto\mathcal{C}|_H$ are dual to each other in the sense
that they are switched by the polarity $\pi$. Simultaneous
point-hyperplane shortening $\mathcal{C}\mapsto\mathcal{C}|_H^P$ glues
these parts together by means of the projection map
$X\mapsto(X+P)\cap H$. The puncturing construction in \cite{etzion2013problems}
is equivalent to $\mathcal{C}\mapsto\mathcal{C}|_H^P$ with the
additional assumption that $P$ and $H$ are not incident. This
assumption implies $\mathcal{C}|^P\cap\mathcal{C}|_H=\emptyset$ and
that $X\mapsto(X+P)\cap H$ maps $\mathcal{C}|^P$ isomorphically onto
the subspace code $\{Y\cap H;Y\in\mathcal{C}|^P\}$.\footnote{Moreover,
  one can show that in this case
  $\mathcal{C}|_H^P\cong\{X+P;X\in\mathcal{C}|_H\}\cup\mathcal{C}|^P$,
  the analogous point-hyperplane shortening using
  $X\mapsto(X\cap H)+P$ instead.}
Shortening in point-hyperplane pairs $(P,H)$ with $P\subseteq H$
seems of little value and will not be considered further in this paper.
\begin{definition}
  \label{def:puncturing}
  Let $\mathcal{C}$ be a subspace code with ambient space $V$, $H$ a
  hyperplane and $P$ a point of $\PG(V)$. 
  The \emph{punctured codes} of
  $\mathcal{C}$ in $H$, $P$ are defined as
  \begin{align*}
    \mathcal{C}_H&=\{X\cap H;X\in\mathcal{C}\},\\
    \mathcal{C}^P&=\bigl\{(X+P)/P;X\in\mathcal{C}\bigr\}
  \end{align*}
  with ambient spaces $H$ and $V/P$, respectively. Moreover, the
  punctured code of $\mathcal{C}$ in $P,H$ with respect to a splitting
  $\mathcal{C}=\mathcal{C}_1\uplus\mathcal{C}_2$ is defined as
  \begin{equation*}
    \mathcal{C}_H^P=(\mathcal{C}_1,\mathcal{C}_2)_H^P=\{X\cap
    H;X\in\mathcal{C}_1\}\cup\{(Y+P)\cap
      H;Y\in\mathcal{C}_2\}
  \end{equation*}
  with ambient space $H$.
\end{definition}
Here mutatis mutandis
the same remarks as on the shortening constructions apply. The original
puncturing operation in \cite{koetter-kschischang08} is
$\mathcal{C}\to\mathcal{C}_H$, with attention restricted to
constant-dimension codes and the following modification: If
$\mathcal{C}$ has constant-dimension $k$ then $\mathcal{C}_H$ can be
turned into a code of constant-dimension $k-1$ by replacing each
subspace $X\in\mathcal{C}$ with $X\cap H=X$ (i.e.\ $X\subseteq H$) by
some ($k-1$)-dimensional subspace contained in $X$. Simultaneous
point-hyperplane puncturing has been defined to round off the
construction principles and will not be used in later sections.

The next lemma provides general information about the parameters of
shortened and punctured subspace codes. The lemma makes reference to
the \emph{degree} of a point $P$ or a hyperplane $H$ with respect to a subspace
code $\mathcal{C}$, which are defined as
$\deg(P)=\{X\in\mathcal{C};P\subseteq X\}=\#\left(\mathcal{C}|^P\right)$
and dually as $\deg(H)=\{X\in\mathcal{C};X\subseteq H\}
=\#\left(\mathcal{C}|_H\right)$,
respectively.\footnote{Incidences with the trivial spaces $\{0\}$, $V$
  (if they are in $\mathcal{C}$) are thus not counted.}
\begin{lemma}
  \label{lma:shortpunct}
  Let $\mathcal{C}$ be a $(v,M,d)_q$ subspace code with ambient space
  $V$ and $(P,H)$ a non-incident point-hyperplane pair in $\PG(V)$.
  \begin{enumerate}[(i)]
  \item\label{l:short} If $d\geq 2$ then the shortened code
    $\mathcal{C}|_H^P$ has parameters $(v-1,M',d')_q$ with
    $M'=\deg(P)+\deg(H)$ and $d'\geq d-1$.
  \item\label{l:punct} If $d\geq 3$ then the punctured codes
    $\mathcal{C}_H$, $\mathcal{C}^P$ have parameters $(v-1,M,d')$ with
    $d'\geq d-2$. The same is true of the punctured code
    $(\mathcal{C}_1,\mathcal{C}_2)_H^P$ with
    respect to any splitting
    $\mathcal{C}=\mathcal{C}_1\uplus\mathcal{C}_2$ satisfying
    $\sdist(\mathcal{C}_1,\mathcal{C}_2)\geq d+1$.
  \end{enumerate}
\end{lemma}
The strong bound $d'\geq d-1$ in Part~\eqref{l:short} 
of the lemma accounts for the significance of
the shortening construction, as mentioned in \cite{etzion2013problems}. 

The usefulness of the bounds in Part~\eqref{l:punct}, which are
weaker, is less clear.  The first assertion in \eqref{l:punct} was
already observed in \cite{koetter-kschischang08} (for the code
$\mathcal{C}_H$). The condition
$\sdist(\mathcal{C}_1,\mathcal{C}_2)\geq d+1$ in the second
assertion is required, since cross-distances can decrease by $3$
during puncturing. Alternatively we could have assumed $d\geq 4$ and
replaced $d'\geq d-2$ by $d'\geq d-3$ in the conclusion.
\begin{proof}[Proof of the lemma]
  \eqref{l:short} 
  Since $P\nsubseteq H$, the codes $\mathcal{C}|_H$ and $\mathcal{C}|^P$
  are disjoint, and since $d\geq 2$, the same is true of
  $\mathcal{C}|_H$ and $\{Y\cap H;Y\in\mathcal{C}|^P\}$.
  Hence we have $\#\mathcal{C}|_H^P
  =\#\left(\mathcal{C}|_H\right)
  +\#\left(\mathcal{C}|^P\right)=\deg(H)+\deg(P)$. 

  Since $Y\mapsto
  Y\cap H$ defines an isometry from $\PG(V/P)$ onto $\PG(H)$, we need
  only check ``cross-distances'' $\sdist(X,Y\cap H)$ with
  $X\in\mathcal{C}|_H$, $Y\in\mathcal{C}|^P$. In this case we have
  \begin{align*}
    \sdist(X,Y\cap H)&=\dim(X)+\dim(Y\cap H)-2\dim(X\cap Y\cap H)\\
    &=\dim(X)+\dim(Y)-1-2\dim(X\cap Y)\\
    &=\sdist(X,Y)-1,
  \end{align*}
  and the assertion regarding $d'$ follows.

  \eqref{l:punct} 
  As in the proof of \eqref{l:short} one shows
  \begin{equation*}
    \sdist(X+P,Y+P)\in
    \begin{cases}
      \{\sdist(X,Y),\sdist(X,Y)-2\}&\text{if $P\nsubseteq X\wedge
        P\nsubseteq Y$}\\
      \{\sdist(X,Y)+1,\sdist(X,Y)-1\}&\text{if $P\nsubseteq X\wedge
        P\subseteq Y$}.
    \end{cases}
  \end{equation*}
  This implies the assertion about $\mathcal{C}^P$, and that about
  $\mathcal{C}_H$ follows by duality. For the last assertion we need
  only check cross-distances $\sdist\bigl(X\cap H,(Y+P)\cap H\bigr)$
  with $X\in\mathcal{C}_1$, $Y\in\mathcal{C}_2$. As shown above, the
  numbers $\sdist(X,Y)$, $\sdist(X,Y+P)$, $\sdist\bigl(X,(Y+P)\cap
  H\bigr)$, $\sdist\bigl(X\cap H,(Y+P)\cap
  H\bigr)$ successively differ by at most one. Hence
  $\sdist\bigl(X\cap H,(Y+P)\cap H\bigr)\geq\sdist(X,Y)-3$, and the
  result follows.\footnote{The distance actually drops by $3$ in the case
    $P\subseteq X+Y\wedge X\cap(P+Y)\subseteq H$, and nontrivial
    examples of $P,H,X,Y$ that satisfy these conditions are easily found.}
\end{proof}



\subsection{A Property of the Lifted Gabidulin Codes}
\label{ssec:gabidulin}

A $q$-ary lifted Gabidulin code $\gabidulin=\gabidulin_{v,k,\delta}$
has parameters $(v,q^{(k-\delta+1)(v-k)},2\delta;k)$, where
$1\leq\delta\leq k\leq v/2$, and can be defined in a coordinate-free
manner as follows (see\ e.g.\ \cite[Sect.~2.5]{smt:fq11proc}): The
ambient space is taken as $V=W\times\F_{q^n}$, where $n=v-k$ and $W$
denotes a fixed $k$-dimensional $\F_q$-subspace of $\F_{q^n}$, and
$\gabidulin$ consists of all subspaces
\begin{equation*}
  G(a_0,\dots,a_{k-\delta})
  =\bigl\{(x,a_0x+a_1x^q+a_2x^{q^2}+\dots+a_{k-\delta}x^{q^{k-\delta}});x\in W\bigr\}
\end{equation*}
with $a_i\in\F_{q^n}$. In other words, $\gabidulin$ consists of the
graphs $\graph_f=\bigl\{(x,f(x));x\in W\bigr\}$ of all
$\F_q$-linear maps $f\colon W\to\F_{q^n}$ that are represented by a
linearized polynomial of symbolic degree at most $k-\delta$.\footnote{Recall
  that every $\F_q$-linear endomorphism of $\F_{q^n}$ is represented
  by a unique linearized polynomial of symbolic degree $\leq
  n-1$. Restriction to $W$ then gives a canonical representation of
  $\F_q$-linear maps $f\colon W\to\F_{q^n}$ by linearized polynomials
  of symbolic degree $\leq k-1$.}

The code $\gabidulin$ forms a geometrically quite regular object. The
most significant property, shared by all lifted MRD codes with the
same parameters, is that $\gabidulin$ forms an exact $1$-cover of the
set of all ($k-\delta$)-flats of $\PG(V)$ that are disjoint from the
special flat $S=\{0\}\times\F_{q^n}$.\footnote{Since the codewords of
  $\gabidulin$ are disjoint from $S$ (since they are graphs of linear
  maps), it is clear that only flats disjoint from $S$ are
  covered. The exact cover property is a consequence of Delsarte's
  characterization of MRD codes (cf.\ Footnote~\ref{fn:delsarte}) and
  is proved in \cite[Lemma~6]{smt:fq11proc}, for example.}  A further
regularity property, which we will need later, is that every point
$P\notin S$ has degree $q^{(k-\delta)(v-k)}$ with respect to
$\gabidulin$. Indeed, $P=\F_q(a,b)\in\graph_f$ if and only if $f(a)=b$
(using $a\neq 0$), which reduces $f$ to a linear map on a
($k-1$)-dimensional subspace of $W$.\footnote{\label{fn:delsarte}Here
  the following property of Gabidulin codes (or lifted MRD codes in
  general) due to Delsarte~\cite{delsarte78a} simplifies the view
  considerably: Every $\F_q$-linear map $g\colon U\to\F_{q^n}$,
  defined on an arbitrary ($k-\delta+1$)-dimensional subspace $U$ of
  $W$, extends uniquely to a linear map $f\in\gabidulin$. This
  property also gives $\#\gabidulin$ immediately.}

From now on we assume that $n=k$ (or $v=2k$, the ``square'' case) and
hence $W=\F_{q^k}$. In this case the codes
$\gabidulin=\gabidulin_{2k,k,\delta}$, $1\leq\delta\leq k$, are 
invariant under a correlation of $\PG(\F_{q^k}\times\F_{q^k})$ fixing
$S$, as our next theorem shows. In particular, every hyperplane $H$ of
$\PG(V)$ with $H\nsupseteq S$ contains, dually so-to-speak, precisely
$q^{(k-\delta)(v-k)}$ codewords of $\gabidulin$, This property will be
needed later in Section~\ref{ssec:d=v-2}. Before stating
the theorem, let us remark that a non-degenerate bilinear form on
$V=\F_{q^k}\times\F_{q^k}$ is given by
$\bigl\langle(a,b),(x,y)\bigr\rangle=\trace(ax+by)$, where
$\trace(x)=\trace_{\F_{q^k}/\F_q}(x)=x+x^q+\dots+x^{q^{k-1}}$ is the
trace of the field extension $\F_{q^k}/\F_q$. The symbol $\perp$
will denote orthogonality with respect to this bilinear form.  Hence
hyperplanes of $\PG(V)$ have the form
$H_{a,b}=\bigl\{(x,y)\in V;\trace(ax+by)=0\bigr\}=\F_q(a,b)^\perp$
for a unique point $\F_q(a,b)$ of $\PG(V)$.

\begin{theorem}
  \label{thm:gabidual}
  The $q$-ary Gabidulin codes $\gabidulin=\gabidulin_{2k,k,\delta}$ are
  linearly isomorphic to their duals
  $\gabidulin^\perp=\{X^\perp;X\in\gabidulin\}$ and hence invariant
  under a correlation of $\PG(v-1,\F_q)$. Any correlation
  $\kappa$ fixing $\gabidulin$ fixes also $S$.
\end{theorem}
\begin{proof}
  The condition $\F_q(a,b)\in G(a_0,\dots,a_{k-\delta})^\perp$, or
  \begin{align*}
    \trace\bigl(ax+bf(x)\bigr)&=\trace\bigl((a+ba_0)x+ba_1x^q+\dots
    +ba_{k-\delta}x^{q^{k-\delta}}\bigr)\\
    &=\trace\bigl((a+ba_0)^{q^{k-\delta}}x^{q^{k-\delta}}+(ba_1)^{q^{k-\delta-1}}
      x^{q^{k-\delta}}+\dots+ba_{k-\delta}x^{q^{k-\delta}}\bigr)\\
    &=0
  \end{align*}
  for all $x\in\F_{q^k}$, is equivalent to
  \begin{equation*}
    a^{q^{k-\delta}}=\sum_{i=0}^{k-\delta}(-a_{k-\delta-i}^{q^i})b^{q^i},
  \end{equation*}
  since the trace bilinear form on $\F_{q^k}$ is non-degenerate. This
  shows
  \begin{equation*}
    \F_q(a,b)\in G(a_0,\dots,a_{k-\delta})^\perp\iff\F_q(b,a^{q^{k-\delta}})\in
    G(-a_{k-\delta},-a_{k-\delta-1}^q,\dots,-a_0^{q^{k-\delta}})    
  \end{equation*}
  In other words, the $\F_q$-linear map $\phi\colon V\to V$,
  $(a,b)\mapsto(b,a^{q^{k-\delta}})$, which represents a collineation
  of $\PG(V)$, maps $G(a_0,\dots,a_{k-\delta})^\perp$ to
  $G(-a_{k-\delta},-a_{k-\delta-1}^q,\dots,-a_0^{q^{k-\delta}})$
  and
  $\gabidulin^\perp=\bigl\{G(a_0,\dots,a_{k-\delta})^\perp;a_i\in\F_{q^k}\bigr\}$
  to $\gabidulin$. The correlation
  $\kappa\colon\F_q(a,b)\mapsto H_{b,a^{q^{k-\delta}}}=\F_q\phi(a,b)^\perp$
  then satisfies $\kappa(\gabidulin)=\gabidulin$, since
  $\trace(ax+by) =\trace(by+a^{q^{k-\delta}}x^{q^{k-\delta}})$ implies
  $\phi(H_{a,b})=H_{b,a^{q^{k-\delta}}}$, i.e.\ $\phi$ and
  $\perp$ commute.\footnote{This also shows that the square of $\kappa$ is the
    collineation induced by
    $\phi^2\colon\F_q(a,b)\mapsto\F_q(a^q,b^q)$.}

  The last assertion follows from the fact that $S$ is the unique
  ($k-1$)-flat complementary to all codewords of $\gabidulin$.
  \end{proof}
  
  \begin{remark}
    \label{rmk:gabidual}
    The automorphism group $\Aut(\gabidulin)$ of
    $\gabidulin=\gabidulin_{2k,k,\delta}$ obviously contains all collineations of
    $\PG(V)$, $V=\F_{q^k}\times\F_{q^k}$, induced by linear maps of the form
    $(x,y)\mapsto\bigl(ax,by+f(x)\bigr)$ with $a,b\in\F_{q^k}^\times$
    and $f(x)$ as above. These
    collineations form a subgroup of $\Aut(\gabidulin)$, which has two
    orbits on the point set $\points$ of $\PG(V)$, viz.\ $S$ and
    $\points\setminus S$.\footnote{Clearly this subgroup is also transitive on
      $\gabidulin$, but this fact won't be used in the sequel.}
    From this and Theorem~\ref{thm:gabidual} we have that for any
    point $P\notin S$ and any hyperplane $H\nsupseteq S$ there exists
    a correlation $\kappa\in\Aut(\gabidulin)$ satisfying $\kappa(P)=H$.
  \end{remark}

\subsection{A Second-Order Bound for the Size of the Union
  of a Finite Set Family}
\label{ssec:bonferroni}

A recurring theme in subsequent sections of this paper will be the
determination of the best lower bound for the size of a union $A_1\cup\dots
\cup A_\delta$ of a family of $t$-sets $A_1,\dots,A_\delta$ with a
prescribed upper bound $s$ on the size of 
their pairwise intersections $A_i\cap A_j$, $1\leq i<j\leq
\delta$. For example, we may ask for the least number of points
covered by $\delta$ $k$-dimensional subspaces of $\F_q^v$ at pairwise
subspace distance $\geq 2k-2$, in which case $t=1+q+\dots+q^{k-1}$ and
$s=1$.\footnote{It goes (well, almost)
  without saying that subspaces of $\F_q^v$ are identified with the
  sets of points ($1$-dimensional subspaces) they contain.} 

Assuming that $A_1,\dots,A_\delta\subseteq A$ for some ``universal''
set $A$, we define $a_i$ as the number of ``points'' $a\in A$
contained in exactly $i$ members of $A_1,\dots,A_\delta$. Our
assumptions and easy double-counting arguments give the following
``standard equations/inequalities'' for the first three binomial
moments of the sequence $a_0,a_1,a_2,\dots$\,:
\begin{equation}
  \label{eq:standard}
  \begin{aligned}
    \mu_0&=\sum_{i\geq 0}a_i=\#A,\\
    \mu_1&=\sum_{i\geq 1} ia_i=\sum_{j=1}^\delta\#A_j=t\delta,\\
    \mu_2&=\sum_{i\geq 2} \binom{i}{2}a_i=\sum_{1\leq j_1<j_2\leq\delta}\#(A_{j_1}\cap
    A_{j_2})\leq s\binom{\delta}{2}.
  \end{aligned}
\end{equation}
From this we infer that $\#(A_1\cup\dots\cup
A_\delta)=\sum_{i\geq 1}a_i\geq\mu_1-\mu_2$ with equality if and only
if $a_i=0$ for $i\geq 4$. This is a special
case of the classical second-order Bonferroni Inequality in
Probability Theory \cite{galambos-simonelli96}. 

Our setting usually permits $a_i>0$ for fairly large $i$, so that we
need a stronger bound. In the following proposition we state a bound
suitable for our purposes. In essence this result is also known from
Probability Theory \cite{galambos-simonelli96}. We have adapted it
to the present combinatorial setting and provide a self-contained
proof,
which may be of independent interest. 

\begin{lemma}
  \label{lma:bonferroni}
  Suppose that $A_1,\dots,A_\delta$ are finite nonempty sets with moments
  $\mu_1,\mu_2$ as defined in \eqref{eq:standard}.\footnote{The
    definition of $\mu_1,\mu_2$ does not depend on the choice of
    $A$.} Then
  \begin{equation*}
    \label{eq:bonferroni}
    \#(A_1\cup\dots\cup
    A_\delta)\quad\geq\frac{2(\mu_1i_0-\mu_2)}{i_0(i_0+1)}
    \quad\text{with $i_0=1+\lfloor 2\mu_2/\mu_1\rfloor$}.
  \end{equation*}
  Moreover, \eqref{eq:bonferroni} remains valid if $\mu_1,\mu_2$ are
  replaced by any known bounds $\overline{\mu}_1\leq\mu_1$ and
  $\overline{\mu}_2\geq\mu_2$.
\end{lemma}
\begin{proof}
  Using the representation $p(X)=p_0+p_1X+p_2X(X-1)/2$, we can
  evaluate $\sum_{i\geq 0}p(i)a_i$ in terms of $\mu_0,\mu_1,\mu_2$
  for every polynomial $p(X)\in\R[X]$ of degree at most $2$.
  If $p(0)>0$ and $p(i)\geq 0$ for $i=1,2,\dots$, this will give an
  upper bound for $a_0$ and hence a lower bound for $\#(A_1\cup\dots\cup
    A_\delta)=\sum_{i\geq 1}a_i=\mu_0-a_0$.\footnote{Exhibiting a quadratic
      polynomial $p(X)$ suitable
    for a particular problem in Combinatorics is sometimes referred to as
    the ``variance trick''; cf.\ \cite[p.~6]{cameron-lint91}.}
  The polynomials
  \begin{equation*}
    p(X)=\binom{X-i}{2}=\frac{1}{2}(X-i)(X-i-1)
    =\binom{X}{2}-i\binom{X}{1}+\frac{i(i+1)}{2},\quad
    i\in\Z^+,
  \end{equation*}
  have this property, since they are convex and vanish
  at two successive integers. We obtain 
  \begin{equation*}
    \frac{i(i+1)}{2}a_0=p(0)a_0\leq\sum_jp(j)a_j
    =\frac{i(i+1)}{2}\mu_0-\mu_1i+\mu_2\quad(i\in\Z^+),
  \end{equation*}
  which is equivalent to
  \begin{equation}
    \label{eq:bonferroni-p1}
    \sum_{j\geq 1}a_j\geq\frac{2(\mu_1i-\mu_2)}{i(i+1)}\quad(i\in\Z^+).
  \end{equation}
  The best (largest) of these bounds can be determined by applying standard
  calculus techniques to the function
  $f(x)=\frac{2(\mu_1x-\mu_2)}{x(x+1)}$, $x>0$. The function $f$ satisfies
  \begin{equation*}
    f\left(\frac{2\mu_2}{\mu_1}\right)=
    f\left(\frac{2\mu_2}{\mu_1}+1\right)=\frac{\mu_1^2}{\mu_1+2\mu_2}
  \end{equation*}
  and $f(x)\lessgtr f(x+1)$ for $x\lessgtr 2\mu_2/\mu_1$. Hence the
  maximum value of $f$ at integral arguments is obtained at
  $i_0=\lfloor 2\mu_2/\mu_1\rfloor+1$.\footnote{If $2\mu_2/\mu_1\in\Z$
    then there are two optimal solutions, viz.\ $2\mu_2/\mu_1$ and 
    $2\mu_2/\mu_1+1$. The (unique) maximum of $f$ is at
    $x=\frac{\mu_2}{\mu_1}\left(1+\sqrt{1+\frac{\mu_1}{\mu_2}}\right)\in
    \left(\frac{2\mu_2}{\mu_1},\frac{2\mu_2}{\mu_1}+1\right)$.}

  For a proof of the last assertion note that the family of bounds
  \eqref{eq:bonferroni-p1} also holds for
  $\overline{\mu}_1,\overline{\mu}_2$ in place of $\mu_1,\mu_2$. 
  The rest of the proof is (mutatis mutandis) the same.
\end{proof}

\section{Classification
  Results for General Parameter Sets}\label{sec:general}

In this section we present old and new results on optimal subspace
codes in the mixed-dimension case for general $q$, $v$, and $d$.  We
start with the largest possible minimum distances (i.e.\ $d\approx v$)
and later switch to small $d$. Whenever possible, we determine the
numbers $\smax_q(v,d)$, the dimension distributions realized by the
corresponding optimal codes, and a classification of the different
isomorphism types. In order to avoid trivialities, we assume from now on
$v\geq 3$ and $2\leq d\leq v$.

\subsection{Subspace Distance $v$}\label{ssec:d=v}

Apart from the trivial case $d=1$ covered already by
Lemma~\ref{lma:smax}\eqref{smax:d=1}, the case $d=v$ is the easiest to
settle. For the statement of Part~\eqref{d=v:even} of the
classification result recall that the largest size of a
$(2k,M,2k;k)_q$ constant-dimension code is $M=\smax_q(2k,2k;k)=q^k+1$
and that optimal $(2k,q^k+1,2k;k)_q$ codes are the same as
($k-1$)-spreads in $\PG(2k-1,\F_q)$, i.e.\ sets of mutually disjoint
($k-1$)-flats (or $k$-dimensional subspaces of $\F_q^{2k}$)
partitioning the point set of $\PG(2k-1,\F_q)$.  The number of
isomorphism classes of such spreads or, equivalently, the number of
equivalence classes of translation planes of order $q^k$ with kernel
containing $\F_q$ under the equivalence relation generated by
isomorphism and transposition
\cite{dembowski68,johnson-jha-biliotti07}, is generally unknown (and
astronomically large even for modest parameter sizes).

\begin{theorem}
  \label{thm:d=v}
  \begin{enumerate}[(i)]
  \item\label{d=v:odd} 
    If $v$ is odd then $\smax_q(v,v)=2$. There are $(v+1)/2$
    isomorphism classes of optimal $(v,2,v)_q$ subspace codes. These have the
    form $\{X,X'\}$ with $\dim(X)=i\in[0,(v-1)/2]$, $\dim(X')=v-i$ and
    $X\cap X'=\{0\}$.
  \item\label{d=v:even}
    If $v=2k$ is even then $\smax_q(v,v)=\smax_q(v,v,k)
    =q^k+1$. Every optimal
    $(v,q^k+1,v)_q$ subspace code has constant dimension $k$.
    The exact number of isomorphism classes of such codes is known in
    the following cases:
    \begin{equation*}
      \begin{array}{c|c|c}
        q&v&\text{\# isomorphism classes}\\\hline
        2&4&1\\
        2&6&1\\
        2&8&7\\
        3&4&2\\
        3&6&7\\
        4&4&3\\
        5&4&20\\
        7&4&973
      \end{array}
    \end{equation*}
  \end{enumerate}
\end{theorem}
The numbers $\smax_q(v,v)$ have also been determined in
\cite[Sect.~5]{gabidulin2009algebraic}. 

\begin{proof}[Proof of the theorem]
  \eqref{d=v:odd} 
  Subspaces $X,X'$ of $V$ are at distance $\sdist(X,X')=v$ iff they are
  complementary ($X\cap X'=\{0\}$, $X+X'=V$). Since $v$ is odd,
  subspaces of the same dimension cannot be complementary, excluding
  the existence of three mutually complementary subspaces. This
  implies $\smax_q(v,v)=2$. The classification of optimal $(v,2,v)_q$
  codes is then immediate.

  \eqref{d=v:even} 
  Suppose that $\mathcal{C}$
  is an arbitrary $(2k,M,2k)_q$ code. If $\mathcal{C}$ contains
  a codeword of dimension $i\neq k$ then all other codewords must have
  dimension $2k-i\neq i$ and hence $\#\mathcal{C}\leq 2$. Certainly
  $\mathcal{C}$ cannot be optimal in this case. Hence
  $\mathcal{C}$ has constant dimension $k$ and size $M=q^k+1$.
  
  Determining the isomorphism classes of the optimal
  $(2k,q^k+1,2k;k)_q$ codes in the table amounts to classifying the
  translation planes of order $\leq 49$ up to isomorphism and
  polarity. This has been done in a series of papers
  \cite{czerwinski-oakden92,dempwolff94,dempwolff-reifart83,hall-swift-walker56,
    mathon-royle95}, from which we have collected the relevant
  information; cf.\ also \cite[Sect.~5]{moorhouse95}.
  \footnote{Uniqueness of the projective planes of orders $4$ and $8$
    gives the uniqueness of the $(4,5,4;2)_2$ and $(6,9,6;3)_2$
    codes. The $8$ translation planes of order $16$ include $1$ polar
    pair (the Lorimer-Rahilly and Johnson-Walker planes), accounting
    for $7$ isomorphism classes of $(8,17,8;4)_2$ codes. The $2$
    translation planes of order $9$ ($\PG(2,\F_3)$ and the Hall plane)
    are both self-polar, accounting for $2$ isomorphism classes of
    $(4,10,4;2)_3$ codes. The $7$ translation planes of order $27$ are
    all self-polar, accounting for $7$ isomorphism classes of
    $(6,28,6;3)_3$ codes. Among the translation planes of order $16$,
    three planes ($\PG(2,\F_{16})$, the Hall plane and one of the two
    semifield planes) have a kernel of order $4$. All three planes are
    self-polar, accounting for $3$ isomorphism classes of
    $(4,17,4;2)_4$ codes. Finally, there are $21$ translation planes
    of order $25$ including $1$ polar pair (the two Foulser planes)
    and $1347$ translation planes of order $49$ including $374$
    polar-pairs, accounting for the remaining two table entries.}
\end{proof}

We remark that it is easy to obtain the numbers $\smax_q(v,v;T)$ for
arbitrary subsets $T\subseteq[0,v]$ from 
Theorem~\ref{thm:d=v}.


\subsection{Subspace Distance $v-1$}\label{ssec:d=v-1}

The case $d=v-1$ is considerably more involved. Here we can no
longer expect that optimal subspace codes have constant dimension,
since for example in a $(2k,q^k+1,2k;k)_q$ constant-dimension
code replacing any codeword by an incident ($k-1$)- or
($k+1$)-dimensional subspace produces a subspace code with
$d=v-1$. However, it turns out that the largest constant-dimension
codes satisfying $d\geq v-1$ are still optimal among all $(v,M,v-1)_q$
codes and that there are only few possibilities for the dimension
distribution of an optimal $(v,M,v-1)_q$ code.

Before stating the classification result for $d=v-1$, let us recall
that in the case of odd length $v=2k+1$ the optimal constant-dimension
codes in the two largest layers $\subspaces{V}{k}$ and
$\subspaces{V}{k+1}$ (which are isomorphic as metric spaces)
correspond to maximal partial ($k-1$)-spreads in $\PG(2k,\F_q)$ and
their duals.\footnote{A partial spread is a set of mutually disjoint
  subspaces of the same dimension which does not necessarily cover
  the whole point set of the geometry.} The 
maximum size of a partial ($k-1$)-spread in $\PG(2k,\F_q)$ is
$q^{k+1}+1$, as determined by Beutelspacher
\cite[Th.~4.1]{beutelspacher75}; cf.\ also
\cite[Th.~2.7]{eisfeld-storme00}. This gives $\smax_q(2k+1,2k;k)_q=
\smax_q(2k+1,2k;k+1)_q=q^{k+1}+1$. Moreover, there are partial spreads
$\mathcal{S}$ of the following type: The $q^k$ holes (uncovered
points) of $\mathcal{S}$ form the complement of a $k$-dimensional
subspace $X_0$ in a ($k+1$)-dimensional subspace $Y_0$, and
$X_0\in\mathcal{S}$. We may call $X_0$ the ``moving subspace'' of
$\mathcal{S}$, since it can replaced by any other $k$-dimensional
subspace of $Y_0$ without destroying the spread property of $\mathcal{S}$.
\begin{theorem}
  \label{thm:d=v-1}
  \begin{enumerate}[(i)]
  \item\label{d=v-1:even}
    If $v=2k$ is even then $\smax_q(v,v-1)=\smax_q(v,v;k)
    =q^k+1$. All optimal
    subspace codes contain, apart from codewords of dimension $k$,
    at most one codeword of each of the dimensions $k-1$ and $k+1$. The 
    dimension distributions realized by optimal subspace codes are
    $(\delta_{k-1},\delta_k,\delta_{k+1})=(0,q^k+1,0)$, $(1,q^k,0)$, 
    $(0,q^k,1)$, $(1,q^k-1,1)$ (and $\delta_t=0$ for all other $t$).
  \item\label{d=v-1:odd} If $v=2k+1\geq 5$ is odd then
    $\smax_q(v,v-1)=\smax_q(v,v-1;k)=q^{k+1}+1$. The dimension
    distributions realized by optimal subspace codes are
    $(\delta_{k-1},\delta_k,\delta_{k+1},\delta_{k+2})=(0,q^{k+1}+1,0,0)$,
    $(0,0,q^{k+1}+1,0)$, $(0,q^{k+1},1,0)$, $(0,1,q^{k+1},0)$,
    $(0,q^{k+1},0,1)$, and $(1,0,q^{k+1},0)$.
  \end{enumerate}
\end{theorem}
In \eqref{d=v-1:odd} it is necessary to exclude the case $v=3$, since
$\smax_q(3,2)=q^2+q+2$; cf.\ Section~\ref{ssec:d=2}.  Some results on
the numbers $\smax_q(v,v-1)$ can also be found in
\cite[Sect.~5]{gabidulin2009algebraic}
\begin{proof}
  \eqref{d=v-1:even} Let $\mathcal{C}$ be an optimal $(2k,M,2k-1)_q$
  code. Since $\subspaces{V}{<k}$ has diameter $2k-2$, at most one
  codeword of dimension $<k$ can occur in $\mathcal{C}$, and similarly
  for dimension $>k$. This and Theorem~\ref{thm:d=v}\eqref{d=v:even}
  give $q^k+1\leq M\leq q^k+3$. Clearly there must be codewords of
  dimension $k$, and hence none of dimensions $<k-1$ or $>k+1$.

  If there exists $X_0\in\mathcal{C}$ with $\dim(X_0)=k-1$ then $X_0\cap
  Z=\emptyset$ for all other codewords $Z\in\mathcal{C}_k$. Similarly,
  if there exists $Y_0\in\mathcal{C}$ with $\dim(Y_0)=k+1$ then
  $Y_0\cap Z=P$ is a point for all $Z\in\mathcal{C}_k$, and if both
  $X_0$ and $Y_0$ exist then they must be complementary subspaces of
  $V$.

  If $\delta_k=q^k+1$ then $\mathcal{C}_k$ is a spread and covers all
  points. Hence $X_0$ cannot exist, i.e.\ $\delta_{k-1}=0$. By duality
  we then have also $\delta_{k+1}=0$ and hence $M=q^k+1$.  Next suppose
  $\delta_k=q^k$. Then $\mathcal{C}_k$ is a partial spread with
  exactly $1+q+\dots+q^{k-1}$ holes (uncovered points). If $X_0$
  exists, it contains $1+q+\dots+q^{k-2}$ of these holes. If $Y_0$
  exists, it contains $\#Y_0-q^k=1+q+\dots+q^{k-1}$ of these holes
  (i.e.\ all holes). Since these two properties conflict with each
  other, we must have $\delta_{k-1}\delta_{k+1}=0$ and hence
  $M=q^k+1$.  The only remaining possibility is $\delta_k =q^k-1$,
  $\delta_{k-1}=\delta_{k+1}=1$; here $M=q^k+1$ as well.
  
  Until now we have shown that $\smax_q(v,v-1;k)=q^k+1$ and the
  dimension distributions realized by optimal codes are among those
  listed. Conversely, the distribution
  $(\delta_{k-1},\delta_k,\delta_{k+1})=(0,q^k+1,0)$ is realized by a
  ($k-1$)-spread, $(1,q^k,0)$ by a ($k-1$)-spread with one subspace
  $X$ replaced by a ($k-1$)-dimensional subspace $X_0\subset X$,
  $(0,q^k,1)$ by the dual thereof, and $(1,q^k-1,1)$ by removing from
  a ($k-1$)-spread a pair of subspaces $X,Y$ and adding $X_0,Y_0$ with
  $\dim(X_0)=k-1$, $\dim(Y_0)=k+1$, $X_0\subset X$, $Y_0\supset Y$ and
  $X_0\cap Y_0=\emptyset$.\footnote{This can be done, since the
    (final) choice of $Y_0$ amounts to selecting a complement to
    $(X_0+Y)/Y$ in $V/Y$, i.e.\ a point outside a hyperplane in the
    quotient geometry $\PG(V/Y)$.}

  \eqref{d=v-1:odd} Let $\mathcal{C}$ be an optimal $(2k+1,M,2k)_q$
  code. From the remarks preceding Theorem~\ref{thm:d=v-1}
  we know that $M\geq q^{k+1}+1$. For reasons of diameter (cf.\ the proof of
  \eqref{d=v-1:even}\footnote{The argument is almost the same: There
    must be codewords of dimension $k$ or $k+1$ (otherwise
    $\#\mathcal{C}\leq 2$), and hence none of dimension $<k-1$ or $>k+2$}),
  $\mathcal{C}$ can contain only codewords of dimensions $k-1$, $k$,
  $k+1$ and $k+2$; moreover, $\delta_{k-1},\delta_{k+2}\leq 1$ and
  $\delta_{k-1}\delta_k=\delta_{k+1}\delta_{k+2}=0$.

  If $\delta_{k-1}=\delta_{k+2}=1$ then $M=2$, which is
  absurd.\footnote{We know already that $M\geq q^{k+1}+1$.} If exactly
  one of $\delta_{k-1},\delta_{k+2}$ is nonzero, we can assume by
  duality that $\delta_{k+2}=1$. The $\delta_k$ codewords in $\mathcal{C}_k$ form
  a partial spread and meet the single codeword
  $Y_0\in\mathcal{C}_{k+2}$ in distinct points. Hence $Y_0$ contains
  $1+q+\dots+q^{k+1}-\delta_k$ holes of $\mathcal{C}_k$ and
  $1+q+\dots+q^{k+1}-\delta_k\leq
  1+q+\dots+q^{2k}-\delta_k(1+q+\dots+q^{k-1})$, the total number of
  holes of $\mathcal{C}_k$. This implies
  $\delta_k\leq q^{k+1}$ and hence $M=q^{k+1}+1$,
  $(\delta_{k-1},\delta_k,\delta_{k+1},\delta_{k+2})=(0,q^{k+1},0,1)$. A
  subspace code realizing this dimension distribution can be obtained
  from a maximal partial ($k-1$)-spread $\mathcal{S}$ of the type
  discussed before Theorem~\ref{thm:d=v-1}, if we replace the moving
  subspace $X_0$ by any ($k+2$)-dimensional subspace
  $Y\supset Y_0$.\footnote{The remaining blocks $X\in\mathcal{S}$
    satisfy $X\cap Y_0=\emptyset$ and hence $\dim(X\cap Y)\leq 1$.}
  
  In the remaining case $\delta_{k-1}=\delta_{k+2}=0$ we may assume
  $\delta_{k+1}\leq\delta_{k}$, again by duality. Assuming further
  $\delta_{k+1}\in\{0,1\}$ or, by symmetry,
  $\delta_{k}\in\{q^{k+1}+1,q^{k+1}\}$ easily leads to $M=q^{k+1}+1$
  and one of the dimension distributions $(0,q^{k+1}+1,0,0)$,
  $(0,q^{k+1},1,0)$. The first distribution is realized by any maximal
  partial ($k-1$)-spread and the second distribution by a subspace code
  obtained from a maximal partial ($k-1$)-spread $\mathcal{S}$ of the
  type discussed before Theorem~\ref{thm:d=v-1}, if we replace the
  moving subspace $X_0$ by $Y_0$.
  
  The only remaining case is $2\leq\delta_{k+1}\leq\delta_{k}\leq
  q^{k+1}-1$ (and $\delta_{k-1}=\delta_{k+2}=0$). Here our goal
  is to show that this forces $\delta_k+\delta_{k+1} <q^{k+1}+1$, a
  contradiction. Since $M=\#\mathcal{C}>q^{k+1}+1$ implies the
  existence of a $(2k+1,q^{k+1}+1,2k)_q$ code, we can
  assume $\delta_{k}=q^{k+1}+1-\delta_{k+1}$ and hence $2\leq\delta_{k+1}
  \leq\left\lfloor(q^{k+1}+1)/2\right\rfloor$.\footnote{Using the
    smaller upper bound for $\delta_{k+1}$, however, does not simplify the
    subsequent proof, and we may just consider the full range
    $2\leq\delta_{k+1}\leq q^{k+1}-1$.}
  
  Let $A$, $B$ be the sets of points covered by $\mathcal{C}_k$ and
  $\mathcal{C}_{k+1}$, respectively. Since codewords in
  $\mathcal{C}_k$ are mutually disjoint and disjoint from those in
  $\mathcal{C}_{k+1}$, we have $A\cap B=\emptyset$,
  $\#A=\delta_k(1+q+\dots+q^{k-1})
  =(q^{k+1}+1-\delta_{k+1})(1+q+\dots+q^{k-1})=1+q+\dots+q^{k-1}+q^{k+1}+\dots+
  q^{2k}-\delta_{k+1}(1+q+\dots+q^{k-1})$,
  and hence $\#B\leq q^k+\delta_{k+1}(1+q+\dots+q^{k-1})$. Further we
  know that codewords in $\mathcal{C}_{k+1}$ intersect each other in
  at most a point. Hence the desired contradiction will follow if we
  can show that the minimum number $c(\delta)$ of points covered by
  $\delta$ subspaces of $V$ of dimension $k+1$ mutually
  intersecting in at most a point is strictly larger than
  than the linear function $g(\delta)=q^k+\delta(1+q+\dots+q^{k-1})$ for all
  $\delta\in\{2,3,\dots,q^{k+1}-1\}$.

  For bounding $c(\delta)$ we use Lemma~\ref{lma:bonferroni}.
  The \emph{degree distribution}
  $(b_0,b_1,b_2,\dots)$ of such a set of subspaces
  is defined by $b_i=\#\{P;\deg(P)=i\}$ and satisfies the
  ``standard equations''
  \begin{align*}
    \sum_{i\geq 0}b_i&=1+q+\dots+q^{2k},\\
    \sum_{i\geq 1} ib_i&=\delta(1+q+\dots+q^k),\\
    \sum_{i\geq 2} \binom{i}{2}b_i&=\binom{\delta}{2}.
  \end{align*}
  In this case the moments $\mu_i=\mu_i(\delta)$, $i=1,2$, 
  and hence also the function $f(x)=f(x;\delta)$ used in the proof
  of Lemma~\ref{lma:bonferroni}, depend on $\delta$. The lemma gives
  \begin{align*}
    c(\delta)=\sum_{i\geq
               1}b_i&\geq\frac{2\delta(1+q+\dots+q^k)i_0-\delta(\delta-1)}
               {i_0(i_0+1)}\\
                    &=\frac{\delta\bigl(1+2i_0(1+q+\dots+q^k)-\delta\bigr)}
                      {i_0(i_0+1)}=f(i_0;\delta)
  \end{align*}
  with $i_0=i_0(\delta)
  =1+\bigl\lfloor\frac{\delta-1}{1+q+\dots+q^k}\bigr\rfloor$,
  which is equivalent to
  \begin{equation}
    \label{eq:simul}
    (i_0-1)(1+q+\dots+q^k)+1\leq\delta\leq i_0(1+q+\dots+q^k).
  \end{equation}
  Since 
  \begin{align*}
    f(i_0;i_0(1+q+\dots+q^k)\bigr)
    &=f\bigl(i_0;1+i_0(1+q+\dots+q^k)\bigr)\\
    &=f\bigl(i_0+1;1+i_0(1+q+\dots+q^k)\bigr),
  \end{align*}
  the latter equality being an instance of
  $f\bigl(\frac{2\mu_2}{\mu_1}\bigr)
  =f\bigl(\frac{2\mu_2}{\mu_1}+1\bigr)$ (cf.\ the proof of
  Lemma~\ref{lma:bonferroni}),
  we see that the lower bound $\delta\mapsto
  f\bigl(i_0(\delta);\delta\bigr)$ has a continuous
  extension $F\colon[1,q^{k+1}]\to\R$, which is strictly concave in the
  intervals displayed in \eqref{eq:simul} and constant on the ``holes''
  of length $1$ in between.
  From this it is clear that we need only check the inequality
  $F(\delta)>g(\delta)$ at the points
  $\delta=(i_0-1)(1+q+\dots+q^k)+1$, $1\leq i_0\leq q$ (left endpoints
  of the intervals in \eqref{eq:simul}). Moreover, at the first
  endpoint $\delta=1$ ($i_0=1$) and the last endpoint
  $\delta=q^{k+1}$ ($i_0=q$) equality is sufficient.\footnote{The last
    point $\delta=q^{k+1}$ is best viewed as the right endpoint of the
    hole $[q^{k+1}-1,q^{k+1}]$, since the function $f(q;\delta)$ not really
    matters. Alternatively, one could check the strict inequality at
    $\delta=2$ and $\delta=q^{k+1}-1$, respectively.}

  In the boundary cases we have indeed equality,
  $F(1)=f(1;1)=q^k+1\cdot(1+q+\dots+q^{k-1})=g(1)$ and
  $F(q^{k+1})=f(q;q^{k+1})=q^k+q^{k+1}(1+q+\dots+q^{k-1})=g(q^{k+1})$,
  as is easily verified from the definition of
  $f(i_0;\delta)$.\footnote{This comes not unexpected, since in
    these cases the optimal codes have size $q^{k+1}+1$ and the sets
    $A,B$ partition the point set of $\PG(2k,\F_q)$.} Finally, for
  $2\leq i_0\leq q-1$ we have
  \begin{align*}
    F\bigl(1+(i_0-1)(1+q+\dots+q^k)\bigr)
    &=\frac{1}{i_0}(1+q+\dots+q^k)\bigl(1+(i_0-1)(1+q+\dots+q^k)\bigr)\\
    &=(1+q+\dots+q^k)\left(1+\tfrac{i_0-1}{i_0}\cdot(q+\dots+q^k)\right)\\
    &>(1+q+\dots+q^k)\left(1+(i_0-1)(1+\dots+q^{k-1})\right)\\
    &=1+\dots+q^k+(i_0-1)(1+\dots+q^k)(1+\dots+q^{k-1})\\
    &=q^k+\bigl(1+(i_0-1)(1+\dots+q^k\bigr)(1+\dots+q^{k-1})\\
    &=g\bigl(1+(i_0-1)(1+q+\dots+q^k)\bigr),
  \end{align*}
  where we have used $\frac{i_0-1}{i_0}\cdot q>i_0-1$. This completes the
  proof of the theorem.\footnote{The last computation could be
    replaced by another convexity argument involving the function
    $i_0\mapsto f\bigl(i_0;1+(i_0-1)(1+q+\dots+q^k)\bigr)$.}
\end{proof}

\begin{remark}
  It is known that any partial ($k-1$)-spread in $\PG(2k-1,\F_q)$ of
  cardinality $q^k-1$ can be completed to a spread; see for example
  \cite[Th.~4.5]{eisfeld-storme00}. This implies that all optimal
  $(2k,q^k+1,2k-1)_q$ subspace codes arise from a ($k-1$)-spread by
  the constructions described at the end of the proof of
  Part~\eqref{d=v-1:even} of Theorem~\ref{thm:d=v-1}. 
\end{remark}

\subsection{Subspace Distance $v-2$}\label{ssec:d=v-2}

The case $d=v-2$ is yet more involved and we are still far from being
able to determine the numbers $\smax_q(v,v-2)$ in general. For even
$v=2k$ the problem almost certainly includes the determination of the
numbers $\smax_q(2k,2k-2;k)$, which are known so far only in a single
nontrivial case, viz.\ $\smax_2(6,4;3)=77$ \cite{smt:fq11proc}. On the
other hand, we will present rather complete information on the odd
case $v=2k+1$, for which the corresponding numbers
$\smax_q(2k+1,2k-1;k)=\smax_q(2k+1,2k;k)=q^{k+1}+1$, equal to the size
of a maximal partial ($k-1$)-spread in $\PG(2k,\F_q)$, are known; cf.\
the references in Section~\ref{ssec:d=v-1}. Our results are collected
in Theorem~\ref{thm:d=v-2} below. For the proof of the theorem we will
need the fact that a maximal partial ($k-1$)-spread $\mathcal{S}$ in
$\PG(2k,\F_q)$ covers each hyperplane at least once. This (well-known)
fact may be seen as follows: If a hyperplane $H$ of $\PG(2k,\F_q)$
contains $t$ members of $\mathcal{S}$, it intersects the remaining
$q^{k+1}+1-t$ members in a ($k-1$)-dimensional space and hence
\begin{equation*}
  t(1+q+\dots+q^{k-1})+(q^{k+1}+1-t)(1+q+\dots+q^{k-2})\leq\#H=1+q+\dots+q^{2k-1},
\end{equation*}
or $tq^{k-1}\leq q^{k-1}+q^{k}$. The difference
$q^{k-1}+q^{k}-tq^{k-1}=q^{k-1}(q+1-t)$ gives the number of holes of
$\mathcal{S}$ in $H$, which must be $\leq q^k$ (the total number of
holes of $\mathcal{S}$). This implies $1\leq t\leq q+1$, as asserted.

We also see that the number of holes of $\mathcal{S}$ in every hyperplane
of $\PG(2k,\F_q)$ is of the form $sq^{k-1}$ with
$s\in\{0,1,\dots,q\}$. Further, since the average number of members of
$\mathcal{S}$ in a hyperplane is
\begin{equation*}
  \frac{(q^{k+1}+1)(1+q+\dots+q^k)}{1+q+\dots+q^{2k}}
  =\frac{1+q+\dots+q^{2k+1}}{1+q+\dots+q^{2k}}>q,
\end{equation*}
there exists at least one hyperplane containing $q+1$ members, and
hence no holes of $\mathcal{S}$. The latter says that the set of holes
of $\mathcal{S}$ does not form a blocking set with respect to
hyperplanes and implies in particular that no line
consists entirely of holes of $\mathcal{S}$. This
fact will be needed later in Remark~\ref{rmk:d=v-2:delta}.

\begin{theorem}
  \label{thm:d=v-2}
  \begin{enumerate}[(i)]
  \item\label{d=v-2:even}
    If $v=2k\geq 8$ is even then 
    $\smax_q(v,v-2)=\smax_q(v,v-2;k)$,
    and the known bound $q^{2k}+1\leq\smax_q(v,v-2;k)\leq(q^k+1)^2$
    applies. Moreover, $\smax_q(4,2)=q^4+q^3+2q^2+q+3$ for all $q$,
    $\smax_2(6,4)=77$ 
    and $q^6+2q^2+2q+1\leq\smax_q(6,4)\leq(q^3+1)^2$ for all
    $q\geq 3$.
  \item\label{d=v-2:odd} If $v=2k+1\geq 5$ is odd then
    $\smax_q(v,v-2)\in\{2q^{k+1}+1,2q^{k+1}+2\}$. Moreover,
    $\smax_q(5,3)=2q^3+2$ 
    for all $q$ and $\smax_2(7,5)=2\cdot 2^4+2=34$.%
    \footnote{The bounds for $\smax_2(v,v-2)$ were already established in
      \cite[Th.~5]{etzion2013problems} and $\smax_2(5,3)=18$ in
      \cite[Th.~14]{etzion2011error}.}
  \end{enumerate}
\end{theorem}
\begin{proof}
  \eqref{d=v-2:even}
  The evaluation of $\smax_q(4,2)$ is a special case of
  Theorem~\ref{thm:d=2} (but could also be easily accomplished
  ad hoc). From now on we assume $k\geq 3$.

  In the constant-dimension case the bounds
  $q^{2k}+1\leq\smax_q(v,v-2;k)\leq(q^k+1)^2$ are well-known; see
  e.g.\ \cite{smt:fq11proc}. In order to show that the upper bound
  holds in the mixed-dimension case as well, let $\mathcal{C}$ be an
  optimal $(2k,M,2k-2)_q$ code and suppose $\mathcal{C}$ contains a
  codeword $X_0$ with $t=\dim(X_0)\neq k$. By duality we may assume
  $t\leq k-1$, and we certainly have $t\geq k-2$, since otherwise
  $\mathcal{C}_{k-1}=\mathcal{C}_k=\emptyset$ and
  $\#\mathcal{C}\leq 1+\smax_q\bigl(2k,2k-2;[k+1,2k]\bigr)
  =1+\smax_q\bigl(2k,2k-2;[0,k-1]\bigr)=1+\smax_q(2k,2k-2;k-1)
  \leq 1+\frac{q^{2k}-1}{q^{k-1}-1}\leq q^{2k}$, contradicting the
  optimality of $\mathcal{C}$. The codewords in
  $\mathcal{C}_{k-2}\cup\mathcal{C}_{k-1}\neq\emptyset$ must be
  mutually disjoint and also disjoint from every codeword in
  $\mathcal{C}_k$. Moreover, $\delta_{k-2}\delta_{k-1}=0$ and
  $\delta_{k-2}\leq 1$.
 
  
  Our strategy now is to bound the size of the ``middle layer''
  $\#\mathcal{C}_k$ in terms of $t=\dim(X_0)$. If this leads to 
  a sharp upper bound for $\mathcal{C}$, which conflicts with
  the best known lower bound for
  $\smax_q(2k,2k-2;k)$, we can conclude $\mathcal{C}=\mathcal{C}_k$,
  and hence $\smax_q(2k,2k-2)=\smax_q(2k,2k-2;k)$.

  Since any two codewords
  in $\mathcal{C}_k$ span at least a ($2k-1$)-dimensional space, we
  have that any ($2k-2$)-dimensional subspace of $V$ contains at most
  one codeword of $\mathcal{C}_k$. Conversely, every codeword of
  $\mathcal{C}_k$, being disjoint from $X_0$, is contained in a
  ($2k-2$)-dimensional subspace intersecting $X_0$ in a subspace of
  the smallest possible dimension, viz.\
  $\max\{t-2,0\}$.\footnote{The condition $\dim(S\cap X_0)=t-2$ is
    of course equivalent to $S+X_0=V$, but we need the former
    for the counting argument.} Denoting by
  $\mathcal{S}$ the set of all such ($2k-2$)-dimensional subspaces and
  by $r$ the (constant) degree of $\#\mathcal{C}_k$ with respect to
  $\mathcal{S}$, we get the bound
  $\#\mathcal{C}_k\leq\#\mathcal{S}/r$. It is easily seen that
  \begin{align*}
    \#\mathcal{S}&=\gauss{t}{t-2}{q}q^{2(2k-t)}=\gauss{t}{2}{q}q^{4k-2t},\\
    r&=\gauss{t}{t-2}{q}q^{2(k-t)}=\gauss{t}{2}{q}q^{2k-2t}
  \end{align*}
  for $t\geq 2$ and hence $\#\mathcal{C}_k\leq q^{2k}$ in this
  case. 

  For $t=1$ (the case $t=0$ does not occur on account of our
  assumption $k\geq 3$) we are in the case $k=3$ and have instead
  $\#\mathcal{S}=(1+q+q^2+q^3+q^4)q^4$, $r=q^2+q$, yielding only the
  weaker bound $\#\mathcal{C}_3\leq\lfloor
  q^3(1+q+q^2+q^3+q^4)/(1+q)\rfloor=q^6+q^4+q^2-q$. However, this
  bound can be sharpened by using for $\mathcal{S}$ the set of
  hyperplanes $H$ not incident with the point $X_0$ and the bound
  $\#(\mathcal{C}\cap H)\leq q^3+1$. The improved bound is
  $\#\mathcal{C}_3\leq(q^3+1)\#\mathcal{S}/r=(q^3+1)q^5/q^2=q^6+q^3$.

  These bounds are sufficient to conclude the proof in the case where
  at most one codeword of dimension $\neq k$ exists. But for the case
  $\delta_{k-1}\geq 2$ and its dual, and for several cases having
  $\delta_{k-2}+\delta_{k-1}=\delta_{k+1}+\delta_{k+2}=1$ we need
  better bounds.

  First we do the case $\delta_{k-1}\geq 2$. Let $X_1,X_2$ be two
  distinct codewords in $\mathcal{C}_{k-1}$. Then $X_1,X_2$ are
  disjoint, and every $X\in\mathcal{C}_k$ is simultaneously disjoint
  from both $X_1$ and $X_2$. In this case we can bound
  $\#\mathcal{C}_k$ in the same way as above, using for $\mathcal{S}$
  the set of ($2k-2$)-dimensional subspaces $S$ of $V$ satisfying
  $\dim(S\cap X_1)=\dim(S\cap X_2)=k-3$. Since the number of
  simultaneous complements of two disjoint lines in $\PG(n-1,\F_q)$ is
  $q^{2n-7}(q^2-1)(q-1)$,\footnote{This is probably well-known and
    perhaps most easily established by counting triples $(L_1,L_2,U)$
    of mutually skew subspaces of $\F_q^n$ with
    $\dim(L_1)=\dim(L_2)=2$, $\dim(U)=n-2$ in two ways: Using
    canonical matrices, the number $\#\bigl\{(L_1,L_2,S)\bigr\}
    =\#\bigl\{(S,L_1,L_2)\bigr\}$ of
    such triples is easily found to be
    $\gauss{n}{n-2}{q}q^{2(n-2)}(q^{n-2}-1)(q^{n-2}-q)$. Dividing this
    number by
    $\#\bigl\{(L_1,L_2)\bigr\}=\gauss{n}{2}{q}\gauss{n-2}{2}{q}q^4$
    gives $q^{2n-7}(q^2-1)(q-1)$, as asserted.}
  we obtain
  \begin{equation*}
    \#\mathcal{S}=\gauss{k-1}{k-3}{q}^2q^{2\cdot 6-7}(q^2-1)(q-1)
                   =\gauss{k-1}{2}{q}^2q^5(q^2-1)(q-1).
  \end{equation*}
  The degree $r$ of $X\in\mathcal{C}_k$ with respect to $\mathcal{S}$
  is equal to the number of ($k-2$)-dimensional subspaces of the
  $k$-dimensional space $V/X$ meeting the (not necessarily distinct)
  hyperplanes $H_1=(X_1+X)/X$ and $H_2=(X_2+X)/X$ in a
  ($k-3$)-dimensional space. By duality, $r$ is also equal to the
  number of lines in $\PG(k-1,\F_q)$ off two points $P_1$, $P_2$
  (which may coincide or not), and
  hence
  \begin{align*}
    r&\geq\gauss{k}{2}{q}-2\gauss{k-1}{1}{q}+1
       =\gauss{k-1}{2}{q}q^2-\left(\gauss{k-1}{1}{q}-1\right)\\
     &=\frac{(q^k-q)(q^{k-1}-q)}{(q^2-1)(q-1)}-\frac{q^{k-1}-q}{q-1}
       =\frac{(q^{k-1}-q)(q^k-q^2-q+1)}{(q^2-1)(q-1)},\\
\#\mathcal{C}_k&\leq\frac{q^4(q^{k-1}-1)^2(q^{k-2}-1)}{q^k-q^2-q+1}
                 =\frac{q^{3k}-q^{2k+2}-2q^{2k+1}+2q^{k+3}+q^{k+2}-q^4}{q^k-q^2-q+1}\\
     &\leq q^{2k}-q^{k+1},
  \end{align*}
  where the last inequality follows from a straightforward
  computation.\footnote{The inequality is sharp precisely in the case
    $k=3$, all $q$.}  This new bound is sufficient for the range
  $2\leq\delta_{k-1}\leq\frac{1}{2}q^{k+1}$ (since we may obviously
  assume $\delta_{k+1}\leq\delta_{k-1}$), but there remains a gap to
  the known upper bound $\delta_{k-1}\leq q^{k+1}+q^2$ for $k\geq 4$,
  respectively, $\delta_2\leq q^4+q^2+1$ for $k=3$. However, for
  $\delta_{k-1}>\frac{1}{2}q^{k+1}$ the standard method to bound
  $\#\mathcal{C}_k$ in terms of the point degrees can be used: Since
  the $\delta_{k-1}(1+q+\dots+q^{k-2})$ points covered by the
  codewords in $\mathcal{C}_{k-1}$ must have degree $0$ in
  $\mathcal{C}_k$, we obtain
  \begin{align*}
    \#\mathcal{C}_k
    &\leq\frac{q^k+1}{q^k-1}\left(q^{2k}-1-\delta_{k-1}(q^{k-1}-1)\right)
    =(q^k+1)^2-\delta_{k-1}\cdot\frac{(q^k+1)(q^{k-1}-1)}{q^k-1}\\
         &<q^{2k}+1-\delta_{k-1}
           \left(\frac{(q^k+1)(q^{k-1}-1)}{q^k-1}-\frac{4}{q}\right).
  \end{align*}
  The factor of $\delta_{k-1}$ is $\geq 2$ in all cases except $q=2$,
  $k=3$, leading to the desired contradiction
  $\#\mathcal{C}\leq\#\mathcal{C}_k+2\delta_{k-1}\leq q^{2k}$. in the
  exceptional case we have $\#\mathcal{C}\leq
  64+1+\frac{1}{7}\delta_{k-1}\leq 68$, which also does the job.


  It remains to consider the cases with
  $\delta_{k-2}+\delta_{k-1} =\delta_{k+1}+\delta_{k+2}=1$.  We may
  assume $k\geq 4$ and need only improve the previously established
  bound $\#\mathcal{C}_k\leq q^{2k}$ by one. We denote the unique
  codewords of dimensions $t<k$ and $u>k$ by $X_0$ and $Y_0$,
  respectively.  From the proof we have $\#\mathcal{C}_k=q^{2k}$ if
  and only if every ($2k-2$)-dimensional subspace of $V$ meeting $X_0$
  in a $t-2$-dimensional space contains a codeword of
  $\mathcal{C}_k$. Since
  $\dim(X_0\cap Y_0) =\frac{1}{2}\bigl(t+u-\sdist(X_0,Y_0)\bigr)
  \leq\left\lfloor\frac{1}{2}\bigl(k-1+k+2-(2k-2)\bigr)\right\rfloor
  =\left\lfloor\frac{3}{2}\right\rfloor=1$,
  there exists a ($2k-2$)-dimensional subspace $S\subset V$ such that
  $\dim(S\cap X_0)=t-2$ and $\dim(S\cap Y_0)\geq u-1$. Then
  $\dim(S+Y_0)\leq 2k-1$ and hence $S+Y_0$, and a fortiori $S$, cannot
  contain a codeword of $\mathcal{C}_k$.\footnote{Note that $X+Y_0=V$
    for every $X\in\mathcal{C}_k$, the dual of $X\cap X_0=\{0\}$.}
  This gives $\#\mathcal{C}_k\leq q^{2k}-1$ and $\#\mathcal{C}\leq
  q^{2k}+1$, as desired.

  Finally, the equality $\smax_2(6,4)=\smax_2(6,4;3)=77$ follows from
  $\smax_2(6,4;3)>2^6+2^3$, which implies $\delta_i=0$ for
  $i\in\{1,2,4,5\}$. The lower bound for $\smax_q(6,4)$, $q\geq 3$
  follows from the corresponding bound for $\smax_q(6,4;3)$,
  established in \cite[Th.~2]{smt:fq11proc}.

  \eqref{d=v-2:odd} First we show $\smax_q(2k+1,2k-1)\geq 2q^{k+1}+1$.
  For this we take the $q$-ary lifted $(2k+2,q^{2(k+1)},2k;k+1)$
  Gabidulin code $\gabidulin=\gabidulin_{2k+2,k+1,k}$, which contains
  $q^{k+1}$ codewords passing through any point $P$ outside the
  special subspace $S=\{0\}\times\F_{q^{k+1}}$ and similarly $q^{k+1}$
  codewords in any hyperplane $H\nsubseteq S$; cf.\
  Theorem~\ref{thm:gabidual}. Among these points and hyperplanes we
  choose a non-incident pair $(P,H)$ and shorten the code $\gabidulin$
  in $(P,H)$; cf.\ Section~\ref{ssec:shortpunct}. The code
  $\mathcal{C}=\gabidulin|_H^P\cup\{H\cap S\}$ then has the required
  parameters $(2k+1,2k-1,2q^{k+1}+1)$.\footnote{The dimension
    distribution of $\mathcal{C}$ is $\delta_k=q^{k+1}+1$,
    $\delta_{k+1}=q^{k+1}$, and $\delta_t=0$ otherwise. It is also
    possible to add a ($k+1$-dimensional space, either through $P$ or
    in $H$, to $\gabidulin$ before
    shortening.}
  Let us remark here that $\mathcal{C}$ admits only extensions
  (without decreasing the minimum distance) by ($k+1$)-dimensional
  subspaces of $\PG(H)$ containing $S\cap H$ (which is $k$-dimensional) and by
  $k$-dimensional subspaces contained in the ($k+1$)-dimensional
  subspace $(S+P)\cap H$. Hence, if $k\geq 3$ then it is impossible to
  extend $\mathcal{C}$ by more than one subspace and improve the
  construction.\footnote{For $k=2$ we can extend by a line in $H$ meeting $S$
    in a point and a plane in $H$ above $S$; cf.\ a subsequent part
    of the proof.}

  Next we establish the upper bound $\smax_q(2k+1,2k-1)\leq
  2q^{k+1}+2$. Let $\mathcal{C}$
  be an optimal $(2k+1,M,2k-1)_q$ code. If $\mathcal{C}$ contains only
  codewords of dimensions $k$ and $k+1$, the bound follows from
  $\smax_q(2k+1,2k;k)=\smax_q(2k+1,2k;k+1)=q^{k+1}+1$. 
  Otherwise we must have $\delta_t=0$ for
  $t\notin\{k-1,k,k+1,k+2\}$, since a codeword of dimension $t\leq
  k-2$ forces $\mathcal{C}_k=\emptyset$, contradicting $M\geq
  2q^{k+1}+1$ (and likewise, using duality, for $t\geq k+3$). The
  remaining cases to consider are
  $(\delta_{k-1},\delta_{k+2})=(1,0)$, $(0,1)$ or $(1,1)$. In these
  cases the bound is established using the remarks preceding the
  theorem. For example, if $X\in\mathcal{C}_{k-1}$ exists then all
  $Y\in\mathcal{C}_{k+1}$ must be disjoint from $X$. Since the dual
  of a maximal partial ($k-1$)-spread in $\PG(2k,\F_q)$ necessarily
  covers every point, this excludes the possibility
  $\delta_{k+1}=q^{k+1}+1$.

  It remains to construct codes meeting the upper bound for
  $k=2$, all $q$ and for $k=3$, $q=2$. 

  First we consider the case $k=2$. Using the shortening construction
  from Section~\ref{ssec:shortpunct}, it suffices to exhibit a
  $(6,2q^3+2,6;3)_q$ constant-dimension code consisting of $q^3+1$
  planes through a point $P$ and $q^3+1$ planes in a hyperplane $H$ of
  $\PG(5,\F_q)$ with $P\notin H$. This can be accomplished by adding
  to the $(6,q^6,4;3)_q$ Gabidulin code two planes $E$, $E'$ with
  $\sdist(E,E')=4$ meeting the special plane $S=\{0\}\times\F_{q^3}$
  in distinct lines $L\neq L'$, respectively, and choose for
  shortening a point $P\in E\setminus S$ and a hyperplane
  $H\supset E'$ with $H\cap S=L'$. Clearly $P,H$ can be taken as
  non-incident, and then shortening the $(6,q^6+2,4;3)_q$ code
  $\gabidulin\cup\{E,E'\}$ in $(P,H)$ yields the desired
  $(5,2q^3+2,3)_q$ code.\footnote{Another 
  construction for $(5,2q^3+2,3)_q$ codes was recently found by 
  Cossidente, Pavese and Storme \cite{new_preprint}.}

In the case $k=3$, $q=2$ we found several $(7,34,5)_2$ codes by a
computer search.\footnote{In fact, we succeeded in a complete
    classification of codes with these parameters.  The total number
    of equivalence classes is $20$.  Details will be given in
    \cite{smt:v7}.} An example is the subspace code given
  by the row spaces of the matrices
  \[\setlength{\arraycolsep}{2pt}
  \begin{array}{ccccc}
\left(\begin{smallmatrix}
1 & 0 & 0 & 0 & 0 & 1 & 1 \\
0 & 1 & 0 & 1 & 1 & 1 & 1 \\
0 & 0 & 1 & 1 & 1 & 0 & 0
\end{smallmatrix}\right)\text{,} &
\left(\begin{smallmatrix}
0 & 1 & 0 & 0 & 1 & 0 & 0 \\
0 & 0 & 1 & 0 & 0 & 1 & 0 \\
0 & 0 & 0 & 1 & 0 & 0 & 1
\end{smallmatrix}\right)\text{,} &
\left(\begin{smallmatrix}
1 & 0 & 0 & 1 & 1 & 0 & 1 \\
0 & 1 & 0 & 0 & 0 & 1 & 1 \\
0 & 0 & 1 & 1 & 1 & 0 & 1
\end{smallmatrix}\right)\text{,} &
\left(\begin{smallmatrix}
0 & 0 & 0 & 0 & 1 & 0 & 0 \\
0 & 0 & 0 & 0 & 0 & 1 & 0 \\
0 & 0 & 0 & 0 & 0 & 0 & 1
\end{smallmatrix}\right)\text{,} &
\left(\begin{smallmatrix}
1 & 0 & 0 & 0 & 1 & 1 & 1 \\
0 & 1 & 0 & 1 & 0 & 1 & 0 \\
0 & 0 & 1 & 1 & 1 & 1 & 1
\end{smallmatrix}\right)\text{,} \\[1.5ex]
\left(\begin{smallmatrix}
1 & 0 & 0 & 0 & 0 & 0 & 0 \\
0 & 0 & 1 & 0 & 0 & 1 & 1 \\
0 & 0 & 0 & 1 & 0 & 1 & 0
\end{smallmatrix}\right)\text{,} &
\left(\begin{smallmatrix}
1 & 0 & 0 & 0 & 0 & 1 & 0 \\
0 & 1 & 0 & 1 & 0 & 0 & 1 \\
0 & 0 & 1 & 0 & 1 & 0 & 1
\end{smallmatrix}\right)\text{,} &
\left(\begin{smallmatrix}
1 & 0 & 0 & 0 & 1 & 1 & 0 \\
0 & 1 & 0 & 1 & 1 & 0 & 0 \\
0 & 0 & 1 & 0 & 1 & 1 & 1
\end{smallmatrix}\right)\text{,} &
\left(\begin{smallmatrix}
1 & 0 & 0 & 0 & 1 & 0 & 1 \\
0 & 1 & 1 & 0 & 0 & 0 & 1 \\
0 & 0 & 0 & 1 & 1 & 0 & 0
\end{smallmatrix}\right)\text{,} &
\left(\begin{smallmatrix}
1 & 0 & 1 & 0 & 0 & 1 & 0 \\
0 & 1 & 0 & 0 & 1 & 0 & 1 \\
0 & 0 & 0 & 1 & 1 & 1 & 0
\end{smallmatrix}\right)\text{,} \\[1.5ex]
\left(\begin{smallmatrix}
1 & 0 & 0 & 0 & 1 & 0 & 0 \\
0 & 0 & 1 & 0 & 0 & 0 & 1 \\
0 & 0 & 0 & 1 & 0 & 1 & 1
\end{smallmatrix}\right)\text{,} &
\left(\begin{smallmatrix}
1 & 0 & 0 & 1 & 1 & 1 & 0 \\
0 & 1 & 0 & 0 & 1 & 1 & 0 \\
0 & 0 & 1 & 0 & 1 & 0 & 0
\end{smallmatrix}\right)\text{,} &
\left(\begin{smallmatrix}
1 & 0 & 0 & 1 & 0 & 0 & 0 \\
0 & 1 & 0 & 0 & 1 & 1 & 1 \\
0 & 0 & 1 & 1 & 1 & 1 & 0
\end{smallmatrix}\right)\text{,} &
\left(\begin{smallmatrix}
0 & 1 & 0 & 0 & 0 & 0 & 0 \\
0 & 0 & 1 & 0 & 0 & 0 & 0 \\
0 & 0 & 0 & 1 & 0 & 0 & 0
\end{smallmatrix}\right)\text{,} &
\left(\begin{smallmatrix}
1 & 0 & 0 & 1 & 0 & 1 & 1 \\
0 & 1 & 0 & 0 & 0 & 1 & 0 \\
0 & 0 & 1 & 0 & 1 & 1 & 0
\end{smallmatrix}\right)\text{,} \\[1.5ex]
\left(\begin{smallmatrix}
1 & 0 & 1 & 0 & 1 & 0 & 0 \\
0 & 1 & 0 & 0 & 0 & 0 & 1 \\
0 & 0 & 0 & 1 & 1 & 1 & 1
\end{smallmatrix}\right)\text{,} &
\left(\begin{smallmatrix}
1 & 0 & 0 & 0 & 0 & 0 & 1 \\
0 & 1 & 1 & 0 & 1 & 1 & 1 \\
0 & 0 & 0 & 1 & 1 & 0 & 1 \\
\end{smallmatrix}\right)\text{,} \\[2ex]
\left(\begin{smallmatrix}
1 & 0 & 0 & 0 & 1 & 1 & 0 \\
0 & 1 & 0 & 0 & 1 & 0 & 0 \\
0 & 0 & 1 & 0 & 1 & 1 & 0 \\
0 & 0 & 0 & 1 & 1 & 1 & 0
\end{smallmatrix}\right)\text{,} &
\left(\begin{smallmatrix}
0 & 1 & 0 & 0 & 0 & 1 & 0 \\
0 & 0 & 1 & 0 & 0 & 0 & 1 \\
0 & 0 & 0 & 1 & 0 & 1 & 0 \\
0 & 0 & 0 & 0 & 1 & 0 & 0
\end{smallmatrix}\right)\text{,} &
\left(\begin{smallmatrix}
1 & 0 & 0 & 0 & 0 & 1 & 1 \\
0 & 1 & 0 & 0 & 0 & 1 & 1 \\
0 & 0 & 0 & 1 & 0 & 1 & 0 \\
0 & 0 & 0 & 0 & 1 & 1 & 1
\end{smallmatrix}\right)\text{,} &
\left(\begin{smallmatrix}
1 & 0 & 0 & 0 & 1 & 1 & 0 \\
0 & 1 & 0 & 0 & 0 & 0 & 0 \\
0 & 0 & 1 & 0 & 0 & 1 & 1 \\
0 & 0 & 0 & 1 & 1 & 0 & 1
\end{smallmatrix}\right)\text{,} &
\left(\begin{smallmatrix}
1 & 0 & 0 & 0 & 0 & 0 & 0 \\
0 & 1 & 1 & 0 & 0 & 0 & 0 \\
0 & 0 & 0 & 1 & 0 & 0 & 1 \\
0 & 0 & 0 & 0 & 0 & 1 & 0
\end{smallmatrix}\right)\text{,} \\[2ex]
\left(\begin{smallmatrix}
1 & 0 & 0 & 0 & 0 & 0 & 1 \\
0 & 1 & 0 & 0 & 1 & 0 & 1 \\
0 & 0 & 1 & 0 & 0 & 0 & 0 \\
0 & 0 & 0 & 1 & 0 & 1 & 1
\end{smallmatrix}\right)\text{,} &
\left(\begin{smallmatrix}
1 & 0 & 0 & 0 & 0 & 1 & 1 \\
0 & 1 & 0 & 0 & 1 & 0 & 1 \\
0 & 0 & 1 & 0 & 1 & 0 & 1 \\
0 & 0 & 0 & 1 & 1 & 1 & 1
\end{smallmatrix}\right)\text{,} &
\left(\begin{smallmatrix}
1 & 0 & 0 & 0 & 0 & 1 & 0 \\
0 & 1 & 0 & 1 & 0 & 0 & 0 \\
0 & 0 & 1 & 1 & 0 & 0 & 1 \\
0 & 0 & 0 & 0 & 1 & 0 & 1
\end{smallmatrix}\right)\text{,} &
\left(\begin{smallmatrix}
1 & 0 & 0 & 0 & 1 & 0 & 1 \\
0 & 1 & 0 & 1 & 0 & 0 & 1 \\
0 & 0 & 1 & 0 & 0 & 0 & 1 \\
0 & 0 & 0 & 0 & 0 & 1 & 0
\end{smallmatrix}\right)\text{,} &
\left(\begin{smallmatrix}
1 & 0 & 0 & 0 & 1 & 0 & 0 \\
0 & 0 & 1 & 0 & 1 & 1 & 0 \\
0 & 0 & 0 & 1 & 0 & 0 & 0 \\
0 & 0 & 0 & 0 & 0 & 0 & 1
\end{smallmatrix}\right)\text{,} \\[2ex]
\left(\begin{smallmatrix}
1 & 0 & 0 & 1 & 0 & 0 & 1 \\
0 & 1 & 0 & 1 & 0 & 1 & 0 \\
0 & 0 & 1 & 0 & 0 & 1 & 0 \\
0 & 0 & 0 & 0 & 1 & 1 & 0
\end{smallmatrix}\right)\text{,} &
\left(\begin{smallmatrix}
1 & 0 & 1 & 0 & 0 & 0 & 0 \\
0 & 1 & 1 & 0 & 1 & 0 & 0 \\
0 & 0 & 0 & 1 & 1 & 0 & 0 \\
0 & 0 & 0 & 0 & 0 & 1 & 1
\end{smallmatrix}\right)\text{,} &
\left(\begin{smallmatrix}
1 & 0 & 0 & 1 & 0 & 1 & 0 \\
0 & 1 & 0 & 0 & 0 & 1 & 0 \\
0 & 0 & 1 & 1 & 0 & 0 & 0 \\
0 & 0 & 0 & 0 & 1 & 0 & 1
\end{smallmatrix}\right)\text{,} &
\left(\begin{smallmatrix}
1 & 0 & 0 & 1 & 1 & 0 & 1 \\
0 & 1 & 0 & 1 & 1 & 0 & 0 \\
0 & 0 & 1 & 1 & 0 & 0 & 1 \\
0 & 0 & 0 & 0 & 0 & 1 & 1
\end{smallmatrix}\right)\text{,} &
\left(\begin{smallmatrix}
1 & 0 & 1 & 0 & 0 & 0 & 1 \\
0 & 1 & 0 & 0 & 0 & 1 & 1 \\
0 & 0 & 0 & 1 & 0 & 0 & 0 \\
0 & 0 & 0 & 0 & 1 & 0 & 0
\end{smallmatrix}\right)\text{,} \\[2ex]
\left(\begin{smallmatrix}
1 & 0 & 0 & 1 & 0 & 0 & 0 \\
0 & 1 & 0 & 0 & 0 & 0 & 1 \\
0 & 0 & 1 & 0 & 0 & 0 & 0 \\
0 & 0 & 0 & 0 & 1 & 1 & 1
\end{smallmatrix}\right)\text{,} &
\left(\begin{smallmatrix}
1 & 0 & 0 & 0 & 0 & 0 & 0 \\
0 & 1 & 0 & 0 & 0 & 0 & 0 \\
0 & 0 & 1 & 1 & 1 & 0 & 0 \\
0 & 0 & 0 & 0 & 0 & 0 & 1
\end{smallmatrix}\right)\text{.}
\end{array}
\]

  The proof of Theorem~\ref{thm:d=v-2} is now complete.
\end{proof}

\begin{remark}
  \label{rmk:d=v-2:even}
  It seems likely that $\smax_q(v,v-2)=\smax_q(v,v-2;k)$, $k=v/2$, holds
  for $v=6$ as well.\footnote{For $q=2$ this is known, as we stated in
    the theorem.} From the proof of Part~\eqref{d=v-2:even} it is
  clear that an optimal $(6,M,4)_q$ code has $\delta_1\leq 1$,
  $\delta_5\leq 1$, $\delta_2=\delta_4=0$, and hence
  $\smax_q(6,4)\leq\smax_q(6,4;3)+2$.  The proof also shows that
  $\smax_q(6,4)>\smax_q(6,4;3)$ requires $\smax_q(6,4;3)\leq q^6+q^3$,
  and hence $\smax_q(6,4)=\smax_q(6,4;3)$ would follow from an
  improved lower bound on $\smax_q(6,4;3)$.
  
  In those cases where $\smax_q(v,v-2)=\smax_q(v,v-2;k)$ one may ask
  whether all optimal codes must have constant dimension. The
  parameter set $(v,d;k)_q=(8,6;4)_2$ illustrates the difficulties in
  answering this question: Presently it is only known that
  $257\leq\smax_2(8,6)=\smax_2(8,6;4)\leq 289$, with the corresponding
  Gabidulin code $\gabidulin=\gabidulin_{8,4,3}$ of size $256$
  accounting for the lower bound. If the true value turns out to be
  $257$ then there are both constant-dimension and mixed-dimension
  codes attaining the bound, since $\gabidulin$ can be extended by any at
  least $2$-dimensional subspace of its special solid $S$; cf.\
  Section~\ref{ssec:gabidulin}.\footnote{Apart from such extensions,
    it is also possible to extend $\gabidulin$ by any $5$- or
    $6$-dimensional space containing $S$ and by a solid meeting $S$
    in a plane. Extensions by more than one codeword are not possible
    (i.e.\ the minimum distance would necessarily be $<6$).}
\end{remark}

\begin{remark}
  \label{rmk:d=v-2:delta}
  The dimension distributions realized by $(2k+1,2q^{k+1}+2,2k-1)_q$
  codes, provided that codes with these parameters actually exist, can
  be completely determined. 

  In the case $k=2$ these are all four
  distributions that have ``survived'' the proof of
  Theorem~\ref{thm:d=v-2}\eqref{d=v-2:odd}, viz.\
  $(0,0,q^3+1,q^3+1,0,0)$, $(0,1,q^3+1,q^3,0,0)$,
  $(0,0,q^3,q^3+1,1,0)$ and $(0,1,q^3,q^3,1,0)$.\footnote{Recall from
    the proof that $\delta_{1}=1$ forces $\delta_{3}\leq q^{3}$, and
    similarly for $\delta_{4}=1$.} This can be seen as follows:

  The shortening construction used in the proof yields a code
  $\mathcal{C}$ in $\PG(H)$ with $\delta_2=\delta_{3}=q^3+1$ and such
  that the layers $\mathcal{C}_2$ and $\mathcal{C}_3$ are (dual)
  partial spreads of the type discussed before
  Theorem~\ref{thm:d=v-1}. Let $E_1,L_1$ be the special plane
  (containing the holes) and the moving line of $\mathcal{C}_2$, and
  $L_2,E_2$ the special line (meet of the dual holes) and moving plane
  of $\mathcal{C}_3$. Then, using the notation in the proof of
  Theorem~\ref{thm:d=v-2}, $E_1=(S+P)\cap H$, $L_1=E\cap H$,
  $L_2=L'$, $E_2=E'$. In other words, $L_1,L_2$ meet in a point
  (the point $L\cap L'\in S$), $E_1=L_1+L_2$,
  and $E_2$ is some other plane through $L_2$.
  Replacing the plane $E_2\in\mathcal{C}$ by any point $Q\in
  L_2\setminus L_1$ minimum distance $3$, since $Q\in E_1\setminus
  L_1$ is a hole of $\mathcal{C}_2$ and $E_2$ is the only plane in
  $\mathcal{C}_3$ containing $Q$,\footnote{For the latter note that
    the points of degree $1$ with respect to $\mathcal{C}_3$ are those
    on $L_2$, since they are contained in $q^2$ dual holes of
  $\mathcal{C}_3$.} and hence 
  gives a $(5,2q^3+2,3)_q$ code with dimension distribution
  $(0,1,q^3+1,q^3,0,0)$. Similarly, replacing $L_1$ by any solid $T$
  containing $E_1$ but not $E_2$ produces 
  a $(5,2q^3+2,3)_q$ code with dimension distribution
  $(0,0,q^3,q^3+1,1,0)$. Finally, since $\sdist(Q,T)=3$, the code
  $\{Q\}\cup(\mathcal{C}_2
  \setminus\{L_1\})\cup(\mathcal{C}_3\setminus\{E_2\})\cup\{T\}$ has
  parameters $(5,2q^3+2,3)_q$ as well and dimension distribution
  $(0,1,q^3,q^3,1,0)$.

  For $k\geq 3$ the only possible dimension distribution is
  $\delta_k=\delta_{k+1}=q^{k+1}+1$. In order to see this, we may
  suppose by duality that
  $(\delta_{k-1},\delta_k,\delta_{k+1},\delta_{k+2})=(1,q^{k+1}+1,q^{k+1},0)$
  or $(1,q^{k+1},q^{k+1},1)$ and must reduce this ad absurdum. In the
  first case, the codeword of dimension $k-1$ must be disjoint from
  the codewords in $\mathcal{C}_k$, which form a maximal partial
  spread in $\PG(2k,\F_q)$. This is impossible, since $k-1\geq 2$ but
  the set of holes of $\mathcal{C}_k$ cannot contain a line. In the
  second case, let
  $X\in\mathcal{C}_{k-1}$, $Y\in\mathcal{C}_{k+2}$ be the unique
  codewords of their respective dimensions, and note that the
  codewords in $\mathcal{C}_{k-1}\cup\mathcal{C}_k$ are mutually
  disjoint and meet $Y$ in at most a point. Since $\delta_k=q^{k+1}$
  is one less than the size of a maximal partial spread,
  $\mathcal{C}_k$ has $1+q+\dots+q^k$ holes. The set of holes contains
  $X$ and at least $\#Y-(q^{k+1}+1)=q+q^2+\dots+q^k$ further points
  from $Y$. This gives the inequality
  \begin{equation*}
    1+2(q^1+q^2+\dots+q^{k-2})+q^{k-1}+q^k\leq 1+q+\dots+q^k,
  \end{equation*}
  which is impossible for $k\geq 3$.
\end{remark}

\begin{question}
  By our preceding considerations, any $(2k+1,2q^{k+1}+2,2k-1)_q$ code
  has the form $\mathcal{C}=\mathcal{S}_1 \cup \mathcal{S}_2^\perp$,
  where $\mathcal{S}_1,\mathcal{S}_2$ are maximal partial
  ($k-1$)-spreads in $\PG(2k,\F_q)$.  To construct a
  $(2k+1,2q^{k+1}+2,2k-1)_q$ code in this way, one has to pick
  $\mathcal{S}_1,\mathcal{S}_2$ in such a way that the intersection of
  any element of $\mathcal{S}_1$ with any element of
  $\mathcal{S}_2^\perp$ is at most a point.  We refer to this as
  \emph{doubling construction}.  The natural question is: Which $q$
  and $k$ admit a doubling construction?

  Apart from $k=2$, where a doubling construction is possible for
  every $q$, the only decided case is $(q,k) = (2,3)$; cf.\
  Theorem~\ref{thm:d=v-2}(\ref{d=v-2:odd}).  On the other hand, by the
  usual interpretation of $q=1$ as the binary Hamming space, the case
  $q=1$ corresponds to a $(2k+1,4,2k-1)$ binary block code.  This code
  exists for $k \leq 2$, but it does not exist for $k\geq 3$.  This
  might be a hint in the direction that the doubling construction is
  not possible for all combinations of $q\geq 2$ and $k$.
\end{question}

\subsection{Subspace distance $2$}\label{ssec:d=2}
\label{subsec_distance_2}

The projective geometry $\PG(v-1,\F_q)$ or, in the vector space view,
the set of $\F_q$-subspaces of $\F_q^v$ under set inclusion,
forms a finite modular geometric
lattice. In particular $\PG(v-1,\F_q)$ is a ranked poset with   
rank function $X\mapsto\dim(X)$. The theory of finite posets
can be used to determine the numbers $\smax_q(v,2)$ and the
corresponding optimal codes, as outlined in \cite{ahlswede2009error}.
The proof uses a result of Kleitman \cite{kleitman1975extremal} on
finite posets with the so-called LYM property, of which the geometries
$\PG(v-1,\F_q)$ are particular examples. Partial results on the numbers
$\smax_q(v,2)$ can also be found in \cite[Sect.~4]{gabidulin2009algebraic}.

The classification of optimal $(v,M,2)_q$ subspace codes is stated as
Theorem~\ref{thm:d=2} below. We will provide a self-contained
proof of the theorem.
The underlying idea is to use information on the 
intersection patterns of a $(v,M,2)_q$ code with the various maximal chains of
subspaces of $\F_q^v$ for a bound on the code size $M$. Recall that a
maximal chain in a poset is a totally ordered subset which is maximal
with respect to set inclusion among all such subsets. The maximal
chains in $\PG(v-1,\F_q)$ have the form
$\mathcal{K}=\{X_0,X_1,\dots,X_v\}$ with $\dim(X_i)=i$ and $X_i\subset
X_{i+1}$.

If we assign to a subspace $X$ as weight $w(X)$ the reciprocal of the number
of maximal chains containing $X$, we can express the code size as
\begin{equation*}
  \#\mathcal{C}=\sum_{X\in\mathcal{C}}1=\sum_{X\in\mathcal{C}}w(X)
\cdot\#\{\mathcal{K};X\in\mathcal{K}\}
=\sum_{\mathcal{K}}\left(\sum_{X\in\mathcal{C}\cap\mathcal{K}}w(X)\right).
\end{equation*}
Since $w(X)=n_i$ depends only on $i=\dim(X)$, the inner sums are all
alike, and it turns out that they attain a simultaneous maximum at
some subspace code, which then of course must be optimal. For other
parameters the same method could in principle be applied using
suitably chosen families of subsets of the lattice $\PG(v-1,\F_q)$,
but it seems difficult to find families producing tight
bounds.\footnote{Usually much is lost through the fact that no
  subspace code can maximize all inner sums simultaneously.}
 


\begin{theorem}
  \label{thm:d=2}
  \begin{enumerate}[(i)]
  \item\label{d=2:even} If $v=2k$ is even then
    \begin{equation*}
      \smax_q(v,2)=\sum_{\substack{0\leq i\leq v\\i\equiv 0\bmod 2}}\gauss{v}{i}{q}.
    \end{equation*}
    The unique (as a set of subspaces) 
    optimal code in $\PG(v-1,\F_q)$ consists of all subspaces $X$ of $\F_q^v$
    with $\dim(X)\equiv k\bmod 2 $, and thus of all even-dimensional
    subspaces for $v\equiv0\bmod 4 $ and of all odd-dimensional
    subspaces for $v\equiv2\bmod 4 $.
  \item\label{d=2:odd} $v=2k+1$ is odd then
    \begin{equation}
      \label{eq:kleitman}
      \smax_q(v,2)=\sum_{\substack{0\leq i\leq v\\i\equiv
          0\bmod 2}}\gauss{v}{i}{q}
      =\sum_{\substack{0\leq i\leq v\\i\equiv
          1\bmod 2}}\gauss{v}{i}{q},
    \end{equation}
    and there are precisely two distinct optimal codes in
    $\PG(v-1,\F_q)$, containing all even-dimensional and all
    odd-dimensional subspaces of $\F_q^v$, respectively. Moreover
    these two codes are isomorphic.
  \end{enumerate}
\end{theorem}
\begin{proof}
  Since the collineation group of
  $\PG(v-1,\F_q)$ is transitive on subspaces of fixed dimension, the
  number $n_i$ of maximal chains through a
  subspace $X$ of dimension $i$ does not depend on the choice of
  $X$. Further, since each maximal chain passes through a unique
  subspace of dimension $i$, we must have $n_i=n/\gauss{v}{i}{q}$,
  where $n$ denotes the total number of maximal chains. Hence
  \eqref{eq:kleitman} can be rewritten as
  \begin{equation}
    \label{eq:kleitman2}
    \#\mathcal{C}=\frac{1}{n}
    \sum_{\mathcal{K}=\{X_0,\dots,X_v\}}
    \left(\sum_{\substack{i\in\{0,\dots,v\}\\
        X_i\in\mathcal{C}}}\gauss{v}{i}{q}\right)
  \end{equation}
  In order to maximize one of the inner sums in \eqref{eq:kleitman2},
  the best we can do (remember the constraint $d\geq 2$) is to choose
  $\mathcal{C}$ such that either
  $\mathcal{C}\cap\mathcal{K}=\{X_0,X_2,X_4,\dots\}$ or
  $\mathcal{C}\cap\mathcal{K}=\{X_1,X_3,X_5,\dots\}$,
  depending on which of the sums
  $\sum_{i\text{ even}}\gauss{v}{i}{q}$,
  $\sum_{i\text{ odd}}\gauss{v}{i}{q}$ is larger.

  For odd $v$ both sums are equal, since
  $\gauss{v}{i}{q}=\gauss{v}{v-i}{q}$. Hence the inner sums are
  maximized by either choice, and for simultaneous maximization
  of all inner sums it is necessary and sufficient
  that the code $\mathcal{C}$ consists either of all even-dimensional
  subspaces or of all odd-dimensional subspaces of
  $\F_q^v$.\footnote{Since $\{0\}$ and $\F_q^v$ belong to all maximal
    chains, the choice of $\mathcal{C}\cap\mathcal{K}$ for one chain
    determines all others.}

  For even $v=2k$ we use that the Gaussian binomial coefficients
  satisfy  
  $\gauss{v}{i}{q}>q\gauss{v}{i-1}{q}$ for $1\leq i\leq k$; cf.\ the
  proof of Lemma~\ref{lma:unimodal}.
  Together with symmetry this implies
  $\gauss{v}{k}{q}>\gauss{v}{k-1}{q} +\gauss{v}{k+1}{q}$ and
  $\gauss{v}{k-2t}{q}+\gauss{v}{k+2t}{q}
  >\gauss{v}{k-2t-1}{q}+\gauss{v}{k+2t+1}{q}$
  for $1\leq t\leq(k-1)/2$. It follows that
  $\sum_{i\equiv k\bmod2}\gauss{v}{i}{q}>\sum_{i\equiv
    k+1\bmod 2}\gauss{v}{i}{q}$
  and that the unique subspace code $\mathcal{C}$ simultaneously
  maximizing all inner sums in \eqref{eq:kleitman2} consists of all
  subspaces $X$ of $\F_q^v$ with $\dim(X)\equiv k\bmod 2$.
\end{proof}

\section{Bounds and Classification Results for Small
  Parameters}\label{sec:special}

In this section we present the best currently known bounds for the
numbers $\smax_2(v,d)$, $v\leq 7$. In most of those cases, where the
numbers $\smax_2(v,d)$ are known, we also provide the classification
of the optimal subspace codes.

Before turning attention to the binary case $q=2$, let us
remark that the results of Section~\ref{sec:general} determine the
numbers $\smax_q(v,d)$ for all $q$ and $v\leq 5$; see
Table~\ref{tbl:smaxq}.
\begin{table}[htbp]
  \centering
  \(
  \begin{array}{|c||c|c|c|c|}
    \hline
    v\backslash d&2&3&4&5\\\hline\hline
    3&q^2+q+2&2&\multicolumn{1}{c}{}&\\\cline{1-4}
    4&q^4+q^3+2q^2+q+3&q^2+1&q^2+1&\\\hline
    5&q^6+q^5+3q^4+3q^3+3q^2+2q+3&2q^3+2&q^3+1&2\\\hline
  \end{array}
  \)\\[1ex]
  \caption{The numbers $\smax_q(v,d)\label{tbl:smaxq}$ for $v\leq 5$} 
\end{table}
Regarding the corresponding classification, we remark that for
$(v,d)\in\{(3,2), (4,2), (5,2)\}$ the optimal codes are unique up to
subspace code isomorphism (cf.\ Theorem~\ref{thm:d=2}); those for
$(v,d)\in\{(3,3),(5,5)\}$ are classified by
Theorem~\ref{thm:d=v}\eqref{d=v:odd};\footnote{The different
  isomorphism types are represented by $\bigl\{\{0\},\F_q^3\bigr\}$
  and a non-incident point-line pair in $\PG(2,\F_q)$, respectively,
  by $\bigl\{\{0\},\F_q^5\bigr\}$, a non-incident point-solid pair and
  a complementary line-plane pair in $\PG(4,\F_q)$.} those for
$(v,d)=(4,3)$ (and essentially for $(v,d)=(4,4)$ as well) have been
classified for $q\leq 7$ as part of the classification of translation
planes of small order (cf.\ Theorems~\ref{thm:d=v}\eqref{d=v:even}
and~\ref{thm:d=v-1}\eqref{d=v-1:even}).

The remainder of this section is devoted exclusively to the case
$q=2$. With a few notable exceptions, the numbers $\smax_2(v,d)$ are
known for $v\leq 7$; see Table~\ref{tbl:smax2}.
 
\begin{table}[htbp]
  \centering
  \(
  \begin{array}{|c||c|c|c|c|c|c|}
    \hline
    v\backslash d&2&3&4&5&6&7\\\hline\hline
    3&8(1)&2(2)&\multicolumn{1}{c}{}&\multicolumn{1}{c}{}&\multicolumn{1}{c}{}
                           &\\\cline{1-4}
    4&37(1)&5(3)&5(1)&\multicolumn{1}{c}{}&\multicolumn{1}{c}{}&\\\cline{1-5}
    5&187(1)&18(24298)&9(7)&2(3)&\multicolumn{1}{c}{}&\\\cline{1-6}
    6&1521(1)&\text{104--118}&77(4)&9(4)&9(1)&\\\hline
    7&14606(1)&\text{593--776}&\text{330--407}&34(20)&17(928)&2(4)\\\hline
  \end{array}
  \)\\[1ex]
  \caption{$\smax_2(v,d)\label{tbl:smax2}$ and isomorphism types
    of optimal codes for $v\leq 7$} 
\end{table}

The exact values in the table come from Section~\ref{sec:general}.
The number of isomorphism types of optimal
$\bigl(v,\smax_2(v,d),d\bigr)_2$ codes is given in parentheses.
The remaining entries in the table will be derived in the
rest of this section, except for the number of isomorphism types
in the two cases $(v,d)=(7,5)$ and $(7,6)$, which will be derived
in a subsequent paper \cite{smt:v7}.

Regarding the classification of the optimal codes corresponding to the
exact values in the table, we
have that the codes for $d=2$ are unique (Theorem~\ref{thm:d=2}) and
those for $d=v\in\{3,5,7\}$ are classified into $2$, $3$ and $4$
isomorphism types, respectively
(Theorem~\ref{thm:d=v}\eqref{d=v:odd}). The codes for $(v,d)=(4,4)$,
$(6,6)$ are unique, since they correspond to the unique line spread in
$\PG(3,\F_2)$, respectively, the unique plane spread in $\PG(5,\F_2)$;
cf.\ Theorem~\ref{thm:d=v}\eqref{d=v:even}. For $(v,d)=(4,3)$ there
are $3$ different isomorphism types, represented by (i) a line spread
$\mathcal{S}=\{L_1,L_2,L_3,L_4,L_5\}$, (ii) the lines
$L_1,L_2,L_3,L_4$ and a point on $L_5$ and (iii) the lines
$L_1,L_2,L_3$, a point $P\in L_4$ and a plane $E\supset L_5$ with
$P\notin E$. For the proof of this assertion we use that by
Theorem~\ref{thm:d=v-1}\eqref{d=v-1:even} the dimension distribution
of an optimal $(4,5,3)_2$ code up to duality must be one of
$(\delta_1,\delta_2,\delta_3)=(0,5,0)$, $(1,4,0)$, $(1,3,1)$ and that
the corresponding codes form a single $\GL(4,\F_2)$-orbit.  The latter
has already been noted for the line spread and can be seen
for the other two configurations as follows: The (element-wise) stabilizer in
$\GL(4,\F_2)$ of $3$ pairwise skew lines $L_1,L_2,L_3$, which may be
taken as the row spaces of $\left(
  \begin{smallmatrix}
    1&0&0&0\\
    0&1&0&0
  \end{smallmatrix}\right)$, $\left(
  \begin{smallmatrix}
    0&0&1&0\\
    0&0&0&1
  \end{smallmatrix}\right)$, $\left(
  \begin{smallmatrix}
    1&0&1&0\\
    0&1&0&1
  \end{smallmatrix}\right)$, is conjugate to the subgroup
formed by all block-diagonal matrices
of the form $\left(
  \begin{smallmatrix}
    \mat{A}&\\
    &\mat{A}
  \end{smallmatrix}\right)$ with $\mat{A}\in\GL(2,\F_2)$.
This subgroup acts regularly on the $6$
points in $L_4\cup L_5$, as is easily verified,\footnote{The $6$
  points have the form $\F_2(\vek{x},\vek{y})$ with
  $\vek{x},\vek{y}\in\F_2^2$ nonzero and distinct, so that 
  regularity follows from the doubly-transitive action of
  $\GL(2,\F_2)$ on $\F_2^2\setminus\{\vek{0}\}$.}
and leaves $\{L_4,L_5\}$ invariant.\footnote{This follows from the
  fact that $L_4\cup L_5$ is uniquely a union of two lines; in
  other words, the spread containing $L_1,L_2,L_3$ is uniquely
  determined.}
Hence the stabilizer of $L_1,L_2,L_3,L_4$ (and $L_5$) acts transitively on
$L_5$, showing uniqueness of the code with dimension distribution
$(1,4,0)$. Moreover, the subgroup of $\GL(4,\F_2)$ fixing
$\{L_1,L_2,L_3\}$ set-wise and the point $P$ is isomorphic to
$\mathrm{S}_3$, and hence there exists $\mat{M}\in\GL(4,\F_2)$
interchanging $L_1$, $L_2$ and fixing $L_3$, $P$.\footnote{Using
  coordinates as above and writing $P=\F_2(\vek{x}|\vek{y})$, 
  we have $\mat{M}=\left(
    \begin{smallmatrix}
      &\mat{A}\\\mat{A}
    \end{smallmatrix}\right)$, where $\mat{A}\in\GL(2,\F_2)$ is
  the ``transposition'' satisfying $\vek{x}\mat{A}=\vek{y}$,
  $\vek{y}\mat{A}=\vek{x}$.}  The matrix $\mat{M}$ cannot fix all
three points on $L_3$ (otherwise it would fix all points in the
plane $L_3+P$ and hence also $L_1,L_2$). Since the three planes
containing $L_5$ are transversal to $L_3$ and one of them (the
plane containing $P$) is fixed by $\mat{M}$, the other two planes
must be switched by $\mat{M}$. This shows that the code with
dimension distribution $(1,3,1)$ is unique as well.

For $(v,d)=(6,5)$ a similar argument shows that there are $4$
isomorphism types of optimal codes, unique codes with dimension distribution
$(\delta_2,\delta_3,\delta_4)=(0,9,0)$, $(1,8,0)$ and two
non-isomorphic codes with distribution $(1,7,1)$. Taking the ambient
space as $\F_8\times\F_8$, $\F_8=\F_2[\alpha]$ with
$\alpha^3+\alpha+1=0$,
the latter two codes can be obtained from
the partial plane spread $\bigl\{\F_8(1,y);y\in\F_8^\times\bigr\}$ by
adding the line $\{0,\alpha,\alpha^2,\alpha^4\}\times\{0\}$ and
either the solid $\F_2\times\F_8$ or $\F_2\alpha^3\times\F_8$.

The remaining cases, which are more complex,
are described in separate subsections. Their analysis often required
computer calculations. The computations for Section~\ref{ssec:v=7,d=4}
were done in SageMath (\texttt{www.sagemath.org}).
For isomorphism checks and the computation of
automorphism groups, we used nauty \cite{mckay-piperno60} and the
algorithm described in~\cite{feulner13} (based on~\cite{feulner09},
see also~\cite{feulner-diss}).  To find subspace codes of maximum
possible size, we used different approaches.

\subsection{The case $(v,d)=(7,4)$}\label{ssec:v=7,d=4}

This case is in a sense the most challenging,
since its exact resolution will very likely encompass an answer
to the existence question for a $2$-analogue of the Fano plane. 
\begin{theorem}
  \label{thm:v=7,d=4}
  We have $330\leq \smax_2(7,4)\le 407$.
\end{theorem} 
The lower bound is realized by adding the whole space $V=\F_2^7$ to
the best currently known $(7,329,4;3)_2$ constant-dimension code
\cite{braun-reichelt14,lt:0328,mt:alb80}.\footnote{Note added in
  proof: Recently a $(7,333,4;3)_2$ constant-dimension code has been
  found \cite{heinlein-kiermaier-kurz-wassermann16}. This improves the
  lower bound for $\smax_2(7,4)$ to $334$.}  If a $2$-analogue of the
Fano plane exists, the same construction yields a $(7,382,4)_2$ code.

Let us remark that the previously best known upper bound was
$\smax_2(7,4)\leq 776$ \cite{bachoc2013bounds}.
The proof of the new upper bound will be divided into a series of
lemmas. Let $\mathcal{C}$ be a $(7,M,4)_2$
subspace code with constant-dimension layers
$(\mathcal{C}_0,\dots,\mathcal{C}_7)$ and
dimension distribution $(\delta_0,\dots,\delta_7)$.
By duality, we may assume w.l.o.g. that $\delta_3\geq\delta_4$.

\begin{lemma}
  \label{lma:k=3,4}
  \begin{enumerate}[(i)]
  \item\label{lma:k=3,4:i} 
  $\delta_4\leq 190$ and $\delta_3\leq f(\delta_4)$ for the function
  $f\colon\{0,1,\dots,190\}\to\Z$ defined by $f(0)=381$ and
\begin{equation*}
   f(\delta)=
   \begin{cases}
     381-\left\lceil\frac{\delta(1+70i_0-\delta)}{7i_0(i_0+1)}\right\rceil
   &\text{if\quad $35(i_0-1)+1\leq\delta\leq 35i_0$, $1\leq i_0\leq 4$},\\
   381-\delta&\text{if\quad $141\leq\delta\leq 190$.}     
 \end{cases}
\end{equation*}
\item\label{lma:k=3,4:ii}
  $\delta_3+\delta_4\leq 381$; equality can hold only in the
constant-dimension case (i.e., for a $2$-analogue of the Fano plane
and its dual).  
\end{enumerate}
\end{lemma}
Part~\eqref{lma:k=3,4:ii} of Lemma~\ref{lma:k=3,4} will not be used in
the sequel, but has been included since it seems interesting in its
own right---showing $\smax_2(7,4;\{3,4\})\le 381$ and characterizing
the corresponding extremal case.
\begin{proof}[Proof of the lemma]
  The codewords in $\mathcal{C}_3$ cover each
  line of $\PG(6,\F_2)$ at most once. Codewords in $\mathcal{C}_4$ may
  cover a line multiple times (up to $9$ times, the size of a maximal
  partial spread in $\PG(4,\F_2)$), but at least they cannot cover the
  same line as a codeword in $\mathcal{C}_3$. Denoting by $c(\delta)$
  the minimum number of lines covered by $\delta$ solids
  $S_1,\dots,S_\delta$ in $\PG(6,\F_2)$ at mutual distance $\geq 4$,
  we have the bound $7\delta_3+c(\delta_4)\leq 7\cdot 381=2667$, the
  total number of lines in $\PG(6,\F_2)$. Thus $\delta_3\leq
  381-\lceil c(\delta_4)/7\rceil$, and any lower bound on
  $c(\delta)$ will yield a corresponding upper bound for
  $\delta_3$.

For lower-bounding $c(\delta)$ we use Lemma~\ref{lma:bonferroni}.
If $b_i$ is the number
of lines contained in exactly $i$ solids then
  \begin{align*}
    \sum_{i\geq 0}b_i&=2667,\\
    \sum_{i\geq 1} ib_i&=35\delta,\\
    \sum_{i\geq 2} \binom{i}{2}b_i&=e\leq\binom{\delta}{2},
  \end{align*}
  where $e$ denotes the number of edges of the distance-$4$ graph of
  $\{S_1,\dots,S_\delta\}$.
  Since the degree of $S_i$ in
  the distance-$4$ graph is at most $7\cdot(21-1)=140$,\footnote{The
    bound is the same as for the distance-$4$ graph of a set of planes
    at mutual distance $\geq 4$.}  a reasonable upper bound for the second
  binomial moment of $b_0,b_1,b_2,\dots$ is
  \begin{equation*}
    \overline{\mu}_2=
    \begin{cases}
      \delta(\delta-1)/2&\text{if $\delta\leq 140$},\\
      70\delta&\text{if $\delta\geq 141$}.
    \end{cases}
  \end{equation*}
Lemma~\ref{lma:bonferroni} gives
\begin{equation*}
   c(\delta)\geq
   \begin{cases}
\frac{\delta(1+70i_0-\delta)}{i_0(i_0+1)}
   &\text{if\quad $35(i_0-1)+1\leq\delta\leq 35i_0$, $1\leq i_0\leq 4$},\\
   7\delta&\text{if\quad $\delta\geq 141$.}     
   \end{cases}
\end{equation*}
Together with the
bound $\delta_3\leq 381-\lceil c(\delta_4)/7\rceil$ this
proves Part~\eqref{lma:k=3,4:i}.

Moreover, since $f(\delta)<381-\delta$ for $\delta\leq 140$ (as is
easily verified), we have the bound $\delta_3+\delta_4\leq 381$;
equality is possible only if $141\leq\delta_4\leq 190$ and the
distance-$4$ graph on $\mathcal{C}_4=\{S_1,\dots,S_{\delta_4}\}$ is
regular of degree $140$.\footnote{Going backwards,
  $\delta_3+\delta_4=381$ implies $c(\delta_4)=7\delta_4$ and
  $e=70\delta_4$.}  But this is absurd, since in the dual (plane) view
every point of $\PG(6,\F_2)$ would be covered by a plane and hence by
exactly $21$ planes (implying $\delta_4=381$). This
proves Part~\eqref{lma:k=3,4:ii}.
\end{proof}
For the next lemma recall that the size of a maximal partial line
spread in $\PG(6,\F_2)$ is $41$ (leaving $127-3\cdot 41=4=q^2$ holes,
the smallest number one can achieve for odd $v$; cf.\
\cite{beutelspacher75,eisfeld-storme00}). This implies $\delta_2\leq
41$ (and, by duality, also $\delta_5\leq 41$).
\begin{lemma}
  \label{lma:k=2,3}
  $\delta_3\leq g(\delta_2)$, where $g\colon\{0,1,\dots,41\}\to\Z$
  takes the following values:
\begin{gather*}
  \setlength{\arraycolsep}{1pt}
    \begin{array}{c||c|c|c|c|c|c|c|c|c|c|c|c|c|c|c|c|c|c}
      \delta&0&1&2&3&4&5&6&7&8&9&10&11&12&13&14&15&16&17\\\hline
      g(\delta)&381&354&328&304&281&260&240&221&203&186&171
                                   &157&145&134&124&115&107&101
    \end{array}\\
    \setlength{\arraycolsep}{1.4pt}
    \begin{array}{c||c|c|c|c|c|c|c|c|c|c|c|c|c|c|c|c|c|c|c|c|c|c|c|c}
      \delta&18&19&20&21&22&23&24&25&26&27&28&29&30&31&32&33&34&35&36&37&38&39&40&41\\\hline
      g(\delta)& 96& 93& 91& 90& 87& 77& 70& 61& 56& 48& 43& 25& 17&
      12& 9& 8& 5& 4& 3& 2& 2& 1& 1& 0
    \end{array}
\end{gather*}
\end{lemma}
\begin{proof}[Proof of the lemma]
  The $\delta_2$ codewords in $\mathcal{C}_2$ form a partial line spread
  in $\PG(6,\F_2)$. No plane $E\in\mathcal{C}_3$ can meet any line
  $L\in\mathcal{C}_2$, since this would result in
  $\sdist(L,E)\in\{1,3\}$. Hence, if $g(\delta)$ denotes an upper
  bound for the number of planes at pairwise subspace distance $\geq
  4$ and disjoint from any set $L_1,\dots,L_\delta$ of $\delta$
  pairwise disjoint lines, we get the bound $\delta_3\leq g(\delta_2)$.

  We have found three such bounds. Each of them turns out to be
  stronger than the other two in a certain subinterval of
  $\{0,1,\dots,41\}$; cf.\ the subsequent Remark~\ref{rmk:k=2,3}.

  The first bound $g_1(\delta)$ is derived from the observation that no 
  line of $\PG(6,\F_2)$ that meets some line $L_i$ can be covered by
  a plane in $\mathcal{C}_3$. Hence, denoting by $m(\delta)$
  any lower bound for the number of such lines, we can set
  $g_1(\delta)=381-\lceil m(\delta)/7\rceil$.

  A good lower bound $m(\delta)$ is obtained as follows. The $\delta$
  lines cover $3\delta$ points. Denoting by
  $b_i$ the number of lines in $\PG(6,\F_2)$ containing exactly $i$ of
  these points, we have
  \begin{equation*}
    \setlength{\arraycolsep}{3pt}
    \begin{array}{rlrlrlrll}
      b_0&+&b_1&+&b_2&+&b_3&=&2667,\\
         &&b_1&+&2b_2&+&3b_3&=&3\delta\cdot 63=189\delta,\\
         &&&&b_2&+&3b_3&=&\binom{3\delta}{2}.
    \end{array}
  \end{equation*}
  Hence
  $b_1+b_2=189\delta-\binom{3\delta}{2}=186\delta-9\binom{\delta}{2}$
  and $b_1+b_2+b_3\geq b_1+b_2+\delta=187\delta-9\binom{\delta}{2}$,
  so that we may take
  $m(\delta)=187\delta-9\binom{\delta}{2}=\delta(383-9\delta)/2$. 
  This gives the first upper bound
  \begin{equation*}
    g_1(\delta)=
      381-\left\lceil\frac{\delta(383-9\delta)}{14}\right\rceil.\footnotemark
  \end{equation*}
\footnotetext{The quadratic $\delta(383-9\delta)$ is decreasing for
$\delta\geq\frac{383}{18}\approx 21$, rendering this bound trivial
for large values of $\delta$.}%
  %
For the second bound we estimate the number $b_0$ of
lines disjoint from $L_1\cup\dots\cup L_\delta$ directly. For this we denote the
complementary point set by $T$ and write $t=\#T=127-3\delta$. Since
any point of $T$ is on at most $\lfloor(t-1)/2\rfloor$ lines that are 
contained in $T$, we have $b_0\leq
(t/3)\lfloor(t-1)/2\rfloor$. Since planes in $\mathcal{C}_3$ can cover
only lines contained in $T$, we obtain our second upper bound
\begin{equation*}
  g_2(\delta)=\left\lfloor\frac{t\lfloor(t-1)/2\rfloor}{21}\right\rfloor,
  \quad t=127-3\delta.
\end{equation*}
For the third bound we estimate the number $c(\delta')$ of points covered by 
$\delta'$ planes at pairwise subspace distance $\geq 4$
with the aid of Lemma~\ref{lma:bonferroni}. The
relevant moments/moment bounds are $\mu_1=7\delta'$,
$\overline{\mu}_2=\delta'(\delta'-1)/2$, resulting in
\begin{equation*}
  c(\delta')\geq\frac{\delta'(14i_0+1-\delta')}{i_0(i_0+1)}\quad\text{for
    $7(i_0-1)+1\leq\delta'\leq 7i_0$}.
\end{equation*}
This leads to our third upper bound
\begin{equation*}
  g_3(\delta)=\max\bigl\{\delta';c(\delta')\leq 127-3\delta\bigr\}.
\end{equation*}
The combination of the three bounds, viz.\
$g(\delta)=\min\bigl\{g_1(\delta),g_2(\delta), 
g_3(\delta)\bigr\}$ for $\delta\in\{0,1,\dots,41\}$, takes exactly the
values shown in the table.
\end{proof}
\begin{remark}
  \label{rmk:k=2,3}
  All three individual bounds are needed for the optimal bound
  $g(\delta)$. In fact, $g(\delta)$ coincides with $g_1(\delta)$ in
  the range $1\leq\delta\leq 21$, with $g_2(\delta)$ in the ranges
  $22\leq\delta\leq 28$, $39\leq\delta\leq 41$, and with $g_3(\delta)$ in the
  range $28\leq\delta\leq 41$.\footnote{Furthermore,
    all bounds take the "trivial"
    value $381$ at $\delta=0$}
\end{remark}

\begin{lemma}
  \label{lma:k=3,5}
  $\delta_3\leq h(\delta_5)$, where $h\colon\{0,1,\dots,41\}\to\Z$
  is the (unique) non-increasing function taking the values
  \begin{equation*}
    \begin{array}{c||c|c|c|c|c|c|c|c|c|c}
      \delta&0&1&2&3&5&9&18&33&37&41\\\hline
      h(\delta)&381&376&372&371&370&369&368&367&366&365
    \end{array}
  \end{equation*}
  and jumps (i.e., $h(\delta)>h(\delta+1)$) precisely at the
  corresponding arguments.\footnote{except for $\delta=41$, of course}
\end{lemma}
\begin{proof}
  First we consider a $4$-flat $U\in\mathcal{C}_5$ and denote by $a_i=a_i(U)$
  the number of planes $E\in\mathcal{C}_3$ intersecting $U$ in a
  subspace of dimension $i$, $1\leq i\leq 3$. Since $\sdist(E,U)\geq
  4$, we must have $a_3=0$, and $a_1,a_2$ are subject to the
  following restrictions:
  \begin{equation}
    \label{eq:k=2,5}
    \setlength{\arraycolsep}{3pt}
    \begin{array}{rlrlr}
      a_1&+&a_2&=&\delta_3,\\
      4a_1&&&\leq&1024,\\
      3a_1&+&6a_2&\leq&1488,\\
      &&a_2&\leq&155.
    \end{array}
  \end{equation}
  This follows from counting the lines $L$ with $\dim(L\cap
  U)=i\in\{0,1,2\}$ covered by $\mathcal{C}_3$ in two ways, observing
  that the number of such lines cannot exceed $2^{10}$ for $i=0$,
  $(63-15)\cdot 31=1488$ for $i=1$, and $\gauss{5}{2}{2}=155$ for
  $i=2$.

  Viewed as a maximization problem for $\delta_3$, \eqref{eq:k=2,5}
  has the unique optimal solution $a_1^*=256$, $a_2^*=120$ with
  objective value $\delta_3^*=376$. This accounts for $h(1)=376$
  in the table. If a plane subspace code with these parameters
  actually exists, it must leave exactly $35$ lines in $U$ uncovered
  (and cover all lines not contained in $U$). 

  Since the $4$-flats in $\mathcal{C}_5$ form a dual partial line
  spread, any two of them intersect in a plane and have at most
  $7$ lines in common. This implies that for $\delta_5>1$ it is
  impossible to leave simultaneously exactly $35$ lines uncovered
  in each $4$-flat in $\mathcal{C}_5$, and enables us
  to derive subsequently a stronger bound for $\delta_3$. 
  
  It is easily checked that a feasible solution of \eqref{eq:k=2,5}
  with $a_2=155-t$ must have $\delta_3\leq 341+t$. For any 
  $t\in\{0,1,\dots,35\}$ either at least $t$ lines remain uncovered in
  each of the $\delta_5$ $4$-flats in $\mathcal{C}_5$ or the bound
  $\delta_3\leq h_1(t)=340+t$ holds. In the first case we will use a lower
  bound $c(\delta,t)$ for the union of $\delta$ line sets of size $t$
  pairwise intersecting in at most $7$ lines (with a slight modification for
  large $\delta$, see below) to get the bound $\delta_3\leq
  h_2(\delta_5,t)$ with $h_2(\delta,t)=381-\lceil
  c(\delta,t)/7\rceil$. Putting the two bounds together and minimizing over $t$
  gives the final bound $\delta_3\leq h(\delta_5)$ with
  \begin{align*}
    h(\delta)&=\min_{0\leq t\leq
               35}\max\bigl\{h_1(t),h_2(\delta,t)\bigr\}\\
    &=381-\max_{0\leq t\leq
               35}\min\bigl\{41-t,\lceil c(\delta,t)/7\rceil\bigr\}.
  \end{align*}
  A suitable bound $c(\delta,t)$ can again be obtained with the aid of
  Lemma~\ref{lma:bonferroni}. Suppose $U_1,\dots,U_\delta$ are $4$-flats in
  $\PG(6,\F_2)$ pairwise intersecting in a plane and
  $\mathcal{L}_1,\dots,\mathcal{L}_\delta$ line sets with
  $\#\mathcal{L}_j=t$ and $L\subset U_j$ for
  $L\in\mathcal{L}_j$. Denote by $b_i$ the number of lines of
  $\PG(6,\F_2)$ contained in exactly $i$ of the line sets. Then 
  \begin{align*}
    \sum_{i\geq 0}b_i&=2667,\\
    \sum_{i\geq 1}ib_i&=t\delta,\\
    \sum_{i\geq 2} \binom{i}{2}b_i&\leq\overline{\mu}_2=7\binom{\delta}{2}.
  \end{align*}
  Applying Lemma~\ref{lma:bonferroni} gives the lower bound
  \begin{equation*}
    \#(\mathcal{L}_1\cup\dots\cup\mathcal{L}_\delta)\geq
    c(\delta,t)=\frac{\delta(2ti_0+7-7\delta)}{i_0(i_0+1)}
    \quad\text{with}\quad
    i_0=1+\left\lfloor\frac{7(\delta-1)}{t}\right\rfloor.
  \end{equation*}
  A different lower bound,
  \begin{equation*}
  \#(\mathcal{L}_1\cup\dots\cup\mathcal{L}_\delta)\geq
  \left\lceil\frac{t\delta}{9}\right\rceil,
  \end{equation*}
  is obtained from the observation that every line
  $L\in\mathcal{L}_1\cup\dots\cup\mathcal{L}_\delta$ can be contained in
  at most $9$ $4$-flats $U_j$, since the $4$-flats $U_j$ containing $L$
  form a dual partial line spread in $\PG(6,\F_2)/L\cong\PG(4,\F_2)$.
  
  Defining $c(\delta,t)$ as the maximum of the two bounds, we obtain 
  the final bound $h(\delta)$ shown in the table.\footnote{Using the second
    bound leads to an improvement of the final bound for
    $34\leq\delta\leq 41$.}
\end{proof}
With these lemmas at hand we can finish the proof of the theorem.

\begin{proof}[Proof of Theorem~\ref{thm:v=7,d=4}]
  It remains to show the upper bound $M\leq 407$.
  First we consider the case $\delta_0=\delta_1=\delta_6=\delta_7=0$.
  Using Lemmas \ref{lma:k=2,3}, \ref{lma:k=3,5} and the corresponding
  bounds in the dual view, we have
  \begin{equation}
    \label{eq:v=7,d=4}
    M\leq\max_{\substack{0\leq\delta_2\leq 41\\
        0\leq\delta_5\leq
        41}}\delta_2+F\Bigl(\min\bigl\{g(\delta_2),h(\delta_5)\bigr\},
    \min\bigl\{g(\delta_5),h(\delta_2)\bigr\}\Bigr)+\delta_5,
  \end{equation}
  where $F(u_3,u_4)$ denotes the maximum of $\delta_3+\delta_4$
  subject to $0\leq\delta_3\leq u_3$, $0\leq\delta_4\leq u_4$,
  $\delta_4\leq\delta_3$ and the constraints $\delta_4\leq 190$,
  $\delta_3\leq f(\delta_4)$ imposed
  by Lemma~\ref{lma:k=3,4}. It is easy to see that 
  \begin{equation*}
    F(u_3,u_4)=\max_{0\leq\delta_4\leq\min\{u_3,u_4,190\}}
    \min\bigl\{u_3,f(\delta_4)\bigr\}+\delta_4.
  \end{equation*}
  Substituting this into \eqref{eq:v=7,d=4} gives a maximization
  problem which can be solved by exhaustive search (taking a few
  seconds on a current PC or notebook). There are two optimal
  solutions, viz.\
  $(\delta_2^*,\delta_3^*,\delta_4^*,\delta_5^*)=(0,365,1,40)$ and
  $(0,365,0,41)$, with objective value
  $\delta_2^*+\delta_3^*+\delta_4^*+\delta_5^*=406$.

  Finally we consider the case where
  $\delta_0+\delta_1+\delta_6+\delta_7>0$. Since clearly
  $\delta_0+\delta_1\leq 1$ and $\delta_6+\delta_7\leq 1$, the proof
  can be finished by determining all solutions of \eqref{eq:v=7,d=4}
  with $M\geq 405$. Apart from the two optimal solutions, these are
  $(\delta_2,\delta_3,\delta_4,\delta_5)=(0,365,1,39)$,
  $(0,365,2,38)$, $(0,366,2,37)$, $(0,366,3,36)$. Since all $6$
  dimension distributions have $\delta_5>0$, corresponding subspace
  codes, provided they exist, can be extended by at most one codeword
  of dimension $1$.
  This gives the bound $M\leq 407$ and completes the proof of the theorem.
\end{proof}
\begin{remark}
  \label{rmk:v=7,k=4}
  The proof of Theorem~\ref{thm:v=7,d=4} shows that a subspace code
  meeting the bound in the theorem must have dimension distribution
  $(0,1,0,365,0,41,0,0)$ or $(0,1,0,365,1,40,0,0)$ (under the
  assumption $\delta_3\geq\delta_4$), and its dimension-three layer
  $\mathcal{C}_3$ must leave one point $P$ of $\PG(6,\F_2)$ uncovered. For
  $\mathcal{C}_3$ we have the bound $\delta_3\leq 372$, since none of
  the $63$ lines through $P$ can be covered by a plane in
  $\mathcal{C}_3$, but this is too weak to exclude the two dimension
  distributions. We can only say the following:
  Should a $(7,407,4)_2$ subspace code indeed exist, then its dimension-three
  layer cannot be extended to a $2$-analogue of the
  Fano plane, since for subcodes of a putative $2$-analogue 
  leaving one point uncovered we have the stronger bound $M\leq 360$.
\end{remark}

\subsection{The Cases $(v,d)=(6,3)$ and $(7,3)$}\label{ssec:d=3)}
Etzion and Vardy obtained the lower bound $\smax_2(6,3) \ge 85$ in
\cite{etzion2011error}, and in \cite{etzion2013problems} Etzion
conjectured that this lower bound could be raised by extending optimal
$(6,M,4;3)_2$ constant-dimension codes. Since $\smax_2(6,4;3)=77$ and
the optimal codes fall into $5$ isomorphism types \cite{smt:fq11proc},
we can easily compute the corresponding cardinality-maximal
extensions:
\begin{itemize}
\item Types A, B: $\#\mathcal{C}\le 91$, with one particular
  realizable dimension distribution $\delta=(0,0,7,77,7,0,0)$;
\item Type C: $\#\mathcal{C}\le 93$, with $\delta=(0,0,8,77,8,0,0)$;
\item Type D: $\#\mathcal{C}\le 95$, with $\delta=(0,0,8,77,10,0,0)$;
\item Type E: $\#\mathcal{C}\le 95$, with $\delta=(0,0,11,77,7,0,0)$.
\end{itemize}
Restricting the allowed dimensions to $\{0,1,2,3\}$, we obtain the following
maximal extensions:
\begin{itemize}
\item Types A, B: $\#\mathcal{C}\le 84$, with $\delta=(0,0,7,77,0,0,0)$;
\item Type C: $\#\mathcal{C}\le 86$, with $\delta=(0,0,9,77,0,0,0)$;
\item Types D, E: $\#\mathcal{C}\le 88$, with $\delta=(0,0,11,77,0,0,0)$.
\end{itemize}
Since $\smax_2(6,3;\{0,1,2\})=\smax_2(6,3;\{4,5,6\})=21$, we have 
$\smax_2(6,3)\leq 77+2\cdot 21-1
=118$. Using an integer linear programming approach, 
combined with some heuristics, we found a $(6,104,3)_2$ code with
dimension distribution $(0,0,17,69,18,0,0)$.

For the case $(v,d)=(7,3)$ we have $\smax_2(7,3)\le 776$
\cite{bachoc2013bounds}.  The previously best known lower bound was
$\smax_2(7,3)\ge 584$ \cite{etzion2013codes}.  Using again an integer linear
programming approach, one of the $(7,329,4;3)_2$ constant-dimension
codes from \cite{mt:alb80} can be extended by at least $262$
four-dimensional codewords. Adding the null space and $\mathbb{F}_2^7$
yields the improved lower bound $\smax_2(7,3)\ge 329+262+1+1=593$.

\subsection{The Case $(v,d)=(5,4)$}\label{ssec:v=5,d=4}
For $v=5$, $d=4$, by Theorem~\ref{thm:d=v-1}
we have to consider the dimension distributions $(0,0,9,0,0,0)$,
$(0,0,8,1,0,0)$ and $(0,0,8,0,1,0)$, up to duality.  The codes with
dimension distribution $(0,0,9,0,0,0)$ are exactly the $(5,9,4;2)_2$
constant dimension codes or in other words, the partial line spreads
of $\PG(4,\F_2)$ of size $9$.  Up to equivalence, there are $4$ such
partial spreads \cite{gordon-shaw-soicher04}; see also
\cite{smt:fq11proc}.

The codes realizing the remaining two dimension distributions contain
a partial line spread $\mathcal{S}_8$ of size $8$ as a subcode. Up to
equivalence, there are $9$ types of $\mathcal{S}_8$, all contained in
some maximal partial line spread $\mathcal{S}_9$ of size $9$
\cite[Sect.~5.2]{gordon-shaw-soicher04}.  For the dimension
distribution $(0,0,8,1,0,0)$, the plane $Y_0$ represented by the unique
codeword of dimension $3$ must be disjoint from each of the $8$ lines
in $\mathcal{S}_8$.  Thus, these codes $\mathcal{C}$ are exactly the
partitions of $V$ into $8$ lines and a single plane.  Extending
$\mathcal{S}_8$ by a line $X_0\subset Y_0$, we get an $\mathcal{S}_9$
having a moving line $X_0$ in the
sense of Section~\ref{ssec:d=v-1} or, using the terminology of
\cite{gordon-shaw-soicher04}, an $\mathcal{S}_9$ of regulus type
\textsf{X} with the $4$ reguli sharing the line $X_0$. Thus the
$\mathcal{S}_8$ contained in $\mathcal{C}$ is the unique regulus-free
partial spread (regulus type \textsf{O} in \cite{gordon-shaw-soicher04}). If
follows that up to equivalence there is a unique code with dimension
distribution $(0,0,8,1,0,0)$.  It is given by the lifted Gabidulin
code $\gabidulin_{5,2,2}$ together with its special plane.

Now let $\mathcal{C}$ be a subspace code with dimension distribution
$(0,0,8,0,1,0)$. It has the form
$\mathcal{C} = \mathcal{S}_8 \cup \{H\}$ with a hyperplane (solid) $H$.
The code
$\mathcal{C}$ has minimum distance $4$ if and only if for each line
$L\in \mathcal{S}$, $\dim(L \cap H) = 1$.
Consider a maximal partial spread $\mathcal{S}_9$
containing $\mathcal{S}_8$. Since $H$ contains at most one line of
$\mathcal{S}_9$, it must be one of the $3$ solids containing the
special plane $Y_0$ of $\mathcal{S}_9$ and contain exactly one line
$L$ of $\mathcal{S}_9$. 
If $L$ is contained in $Y_0$, $\mathcal{S}_8$ has regulus type
\textsf{O} and $\mathcal{C} = \mathcal{S}_8 \cup \{H\}$ has the type
mentioned in the proof of Theorem~\ref{thm:d=v-1}\eqref{d=v-1:odd}
with $H=Y$.  Moreover, it is readily checked that the $3$ possible
choices for $Y$ yield equivalent codes.
If $L$ is not contained in $Y_0$, then $H=L+Y_0$ and $L$ is contained
in $2$ reguli of $\mathcal{S}_9$. This implies that $\mathcal{S}_8$
has regulus type \textsf{II} and is again uniquely determined
\cite{gordon-shaw-soicher04}. Since $H$ is determined by
$\mathcal{S}_8$ (for example, as the span of the $7$ holes of
$\mathcal{S}_8$), $\mathcal{C}$ is uniquely determined as well.


In all we have seen that up to equivalence there are $2$ subspace codes
realizing the dimension distribution $(0,0,8,0,1,0)$.

Altogether, there are $4 + 1 + 2 = 7$ types of $(5,9,4)_2$ subspace codes.

\subsection{The Case $(v,d)=(5,3)$}\label{ssec:v=5,d=3)}

For $v=5$, $d=3$, by Remark~\ref{rmk:d=v-2:delta} we have to
consider the dimension distributions $(0,0,9,9,0,0)$,
$(0,1,8,8,1,0)$ and $(0,0,8,9,1,0)$, up to duality.  So in each
case, there is a subcode of dimension distribution $(0,0,9,0,0,0)$
or $(0,0,8,0,1,0)$ and subspace distance $4$ (all dimensions of the
codewords in the subcode have the same parity, so distance $3$
implies distance $4$).  We have already seen that up to equivalence,
the number of possibilities for this subcode is $4$ or $2$,
respectively.  For these $6$ starting configurations, we enumerated
all extensions to a code of size $18$ by a clique search
\cite{cliquer}.  The resulting codes have been filtered for
equivalence using nauty.  In the end, we got the following numbers
of equivalence classes: For $\delta = (0,0,9,9,0,0)$, there are
17708 codes, among them 306 self-dual ones.  For
$\delta = (0,1,8,8,1,0)$, there are 2164 codes, among them 73
self-dual ones.  For $\delta = (0,0,8,9,1,0)$, there are 4426 codes,
of course none of them self-dual.  In total, there are
$17708+2164+4426 = 24298$ types of $(5,9,3)_2$ subspace codes.


\section*{Acknowledgments} 

The authors are grateful to Tuvi Etzion for comments regarding the
upper bound in Theorem~\ref{thm:v=7,d=4} and to two anonymous
reviewers for valuable comments and corrections.

The first author would like to thank the organizers of 
ALCOMA~15 for the invitation to the conference, giving him the opportunity to
spend a very pleasant and academically rewarding week at Kloster Banz.

\bibliographystyle{AIMS} 

\end{document}